\RequirePackage{fix-cm}
\documentclass[11pt, reqno]{amsart}
\usepackage{amsmath}
\usepackage{setspace}
\usepackage{fancyhdr}
\usepackage{mathrsfs}
\usepackage{xifthen}
\usepackage{verbatim}
\usepackage{hyperref}
\usepackage[all]{xcolor}

\usepackage{float}

\usepackage{comment}

\usepackage{geometry}
\usepackage{fancyhdr}
\usepackage{textcomp}
\usepackage{amsthm}
\usepackage{amstext}
\usepackage{amsfonts}
\usepackage{amssymb}
\usepackage{fixltx2e}
\usepackage{graphicx}

\usepackage{tikz}
\usepackage{tikz-cd}
\usepackage{caption}
\usetikzlibrary{calc,backgrounds}
\usetikzlibrary{matrix,arrows,decorations.pathmorphing}

\usepackage{bigstrut}

\makeatletter

\definecolor{mynicegreen}{RGB}{52,194,48}
\definecolor{mydarkteal}{RGB}{0,102,102}

\newcommand{\schff}{X}
\newcommand{\adicff}{\mathcal{X}}
\newcommand{\Spa}{\mathrm{Spa}}

\newcommand{\HN}{\mathrm{HN}}

\newcommand{\rk}{\mathrm{rk}}

\newcommand{\inj}{\hookrightarrow}
\newcommand{\surj}{\twoheadrightarrow}

\newcommand{\mumax}{\mu_\text{max}}
\newcommand{\mumin}{\mu_\text{min}}
\newcommand{\strsheaf}[1][]{\Ocal_{#1}}
\newcommand{\trivbundle}{\strsheaf}

\newcommand{\genred}{\overline}
\newcommand{\gensecred}[1]{{#1}'}
\newcommand{\genthirdred}[1]{{#1}''}

\newcommand{\genpolygon}{\mathscr{P}}
\newcommand{\gensecpolygon}{\mathscr{Q}}
\newcommand{\HNslope}[2]{\mu_{#2}(#1)}

\newcommand{\gensubvb}{\Dcal}
\newcommand{\genquotvb}{\Fcal}
\newcommand{\genextvb}{\Ecal}
\newcommand{\genvb}{\Vcal}
\newcommand{\gensecvb}{\Wcal}
\newcommand{\auxsubvb}{\Kcal}

\newcommand{\uniformizer}{\pi}
\newcommand{\pseudounif}{\varpi}
\newcommand{\genlocfield}{E}
\newcommand{\finext}{E}
\newcommand{\FFfield}{F}
\newcommand{\algclosedperfdfield}{F}

\newcommand{\integerring}[1]{\strsheaf[#1]}
\newcommand{\witt}{W}
\newcommand{\genfrob}{\phi}

\newcommand{\numslope}{r}
\newcommand{\auxindex}{l}

\numberwithin{equation}{section}

\usepackage{color}

\newcommand{\F}{\mathbb{F}}
\newcommand{\Q}{\mathbb{Q}}
\newcommand{\Z}{\mathbb{Z}}

\newcommand{\PP}{\mathbb{P}}

\DeclareMathOperator{\GL}{GL}

\DeclareMathOperator{\coker}{coker}

\DeclareMathOperator{\Spec}{Spec\,}

\DeclareMathOperator{\Ext}{Ext}
\DeclareMathOperator{\Hom}{Hom}
\DeclareMathOperator{\Pic}{Pic}

\DeclareMathOperator{\Proj}{Proj\,}

\DeclareMathOperator{\dR}{dR}

\DeclareFontFamily{OT1}{rsfs}{}
\DeclareFontShape{OT1}{rsfs}{n}{it}{<-> rsfs10}{}
\DeclareMathAlphabet{\mathscr}{OT1}{rsfs}{n}{it}

\newcommand{\Dcal}{\mathcal{D}}
\newcommand{\Ecal}{\mathcal{E}}
\newcommand{\Fcal}{\mathcal{F}}

\newcommand{\Kcal}{\mathcal{K}}

\newcommand{\Ocal}{\mathcal{O}}

\newcommand{\Vcal}{\mathcal{V}}
\newcommand{\Wcal}{\mathcal{W}}

\newcommand{\Ycal}{\mathcal{Y}}

\newtheorem{lemma}[subsection]{Lemma}
\newtheorem{prop}[subsection]{Proposition}

\theoremstyle{remark}
\newtheorem*{remark}{Remark}
\newtheorem{defn}[subsection]{Definition}
\newtheorem{example}[subsection]{Example}

\newtheorem*{thm*}{Theorem}
\makeatletter
\def\th@remark{%
  \thm@headfont{\bfseries}%
  \normalfont 
}
\def\imod#1{\allowbreak\mkern5mu({\operator@font mod}\,\,#1)}
\makeatother

\usepackage{enumitem}		
\theoremstyle{theorem}
\newtheorem{theorem}[subsection]{Theorem}

\numberwithin{equation}{section}

\makeatother

\begin{document}
	
	\tikzset{
		node style sp/.style={draw,circle,minimum size=\myunit},
		node style ge/.style={circle,minimum size=\myunit},
		arrow style mul/.style={draw,sloped,midway,fill=white},
		arrow style plus/.style={midway,sloped,fill=white},
	}
    
	\title{Extensions of vector bundles on the Fargues-Fontaine curve II}
   
    \author[S. Hong]{Serin Hong}
    \address{Department of Mathematics, University of Arizona, 617 N Santa Rita Ave, Tucson, AZ 85721}
    \email{serinh@math.arizona.edu}
    
    \begin{abstract} Given two arbitrary vector bundles on the Fargues-Fontaine curve, we completely classify all vector bundles which arise as their extensions. 

    \end{abstract}
	
	\maketitle

	\tableofcontents
	
	\rhead{}

	\chead{}
\section{Introduction}

The \emph{Fargues-Fontaine curve} is a geometric object that plays a fundamental role in many parts of arithmetic geometry. It was originally constructed by Fargues-Fontaine \cite{FF_curve} as a main tool for a geometric formulation of classical $p$-adic Hodge theory. The work of Scholze-Weinstein \cite{SW_berkeley} then used it to develop new technical frameworks for $p$-adic geometry, including local Shimura varieties and the $B_{\dR}^+$-Grassmanninans. Recently, the seminal work of Fargues-Scholze \cite{FS_geomLL} established a semi-simplified local Langlands correspondence for a general $p$-adic group $G$ in terms of the stack of $G$-bundles on the Fargues-Fontaine curve.

In this article, we investigate the question of classifying all vector bundles on the Fargues-Fontaine curve that arise as extensions of two given vector bundles. This question naturally arises in the study of $p$-adic flag varieties and the $B_{\dR}^+$-Grassmanninans, as in the work of Viehmann \cite{Viehmann_weakadmlocNewton} and Chen-Tong \cite{CT_weakaddandnewton}. 
Our main result gives a complete answer to this question, extending the previous work of the author \cite{Hong_extvbss}. In addition, our main result leads to a classification of all nonempty Newton strata in minuscule $p$-adic flag varieties for $\GL_n$ by the subsequent work of the author \cite{Hong_generalnewtonstrata}. We hope that our result will further lead to a concise description of how the Harder-Narasimhan strata and the Newton strata on the $B_{\dR}^+$-Grassmanninans intersect, building upon the work of Shen \cite{Shen_HNstrata}, Viehmann \cite{Viehmann_weakadmlocNewton}, and Nguyen-Viehmann \cite{NV_HNstrata}.

For a precise statement of our main result, we set up some basic notations and collect some fundamental facts about the Fargues-Fontaine curve. Let $\genlocfield$ be a nonarchimedean local field with finite residue field $\F_q$, where $q$ is a power of a fixed prime number $p$, and let $\FFfield$ an algebraically closed nonarchimedean complete field of characteristic $p$. We denote by $\schff_{\genlocfield, \FFfield}$ the Fargues-Fontaine curve associated to the pair $(\genlocfield, \FFfield)$, which is a regular noetherian scheme over $\genlocfield$ of Krull dimension $1$. The Picard group of $\schff_{\genlocfield, \FFfield}$ is naturally isomorphic to $\Z$, and thus yields 
a good Harder-Narasimhan formalism for vector bundles on $\schff_{\genlocfield, \FFfield}$. The main result of Fargues-Fontaine \cite{FF_curve} states that 
the Harder-Narasimhan filtration of an arbitrary vector bundle $\genvb$ on $\schff_{\genlocfield, \FFfield}$ splits. In other words, the isomorphism class of every vector bundle $\genvb$ on $\schff_{\genlocfield, \FFfield}$ is determined by its Harder-Narasimhan polygon $\HN(\genvb)$. We regard 
$\HN(\genvb)$ as (the graph of) a concave piecewise linear function with the left endpoint at the origin, and denote by $\HNslope{\HN(\genvb)}{i}$ the slope of $\HN(\genvb)$ on the interval $[i-1, i]$ for each integer $i > 0$. 

\begin{theorem}\label{classification of extensions, semistable case intro}
Let $\gensubvb$, $\genextvb$ and $\genquotvb$ be vector bundles on $\schff_{\genlocfield, \FFfield}$ 
such that $\gensubvb$ or $\genquotvb$ is semistable. 
There exists a short exact sequence 
\[ 0 \longrightarrow \gensubvb \longrightarrow \genextvb \longrightarrow \genquotvb \longrightarrow 0\]
if and only if the line segments of $\HN(\gensubvb \oplus \genquotvb)$ can be rearranged so that the resulting (possibly non-concave) polygon $\genpolygon$ satisfies the following properties:
\begin{enumerate}[label=(\roman*)]
\item $\genpolygon$ lies above $\HN(\genextvb)$ with the same endpoints. 
\smallskip

\item 
For $i = 1, \cdots, rank(\genextvb)$, the polygon $\genpolygon$ has a constant slope $\HNslope{\genpolygon}{i}$ on $[i-1, i]$ with
\smallskip
\begin{itemize}
\item $\HNslope{\genpolygon}{i} < \HNslope{\HN(\genextvb)}{i}$ only if $\HNslope{\genpolygon}{i}$ occurs as a slope in $\HN(\gensubvb)$, 
\smallskip

\item $\HNslope{\genpolygon}{i} > \HNslope{\HN(\genextvb)}{i}$ only if $\HNslope{\genpolygon}{i}$ occurs as a slope in $\HN(\genquotvb)$.
\end{itemize}
\end{enumerate}
\end{theorem}
\begin{figure}[H]
\begin{tikzpicture}[scale=0.5]	
		\pgfmathsetmacro{\eonex}{2}
		\pgfmathsetmacro{\eoney}{3}
		\pgfmathsetmacro{\etwox}{3.5}
		\pgfmathsetmacro{\etwoy}{1}
		\pgfmathsetmacro{\ethreex}{4.5}
		\pgfmathsetmacro{\ethreey}{-2.5}
		
		\pgfmathsetmacro{\fonex}{1.5}		
		\pgfmathsetmacro{\foney}{5}		
		\pgfmathsetmacro{\ftwox}{2.5}
		\pgfmathsetmacro{\ftwoy}{2.5}
		\pgfmathsetmacro{\fthreex}{2}
		\pgfmathsetmacro{\fthreey}{0}
		\pgfmathsetmacro{\ffourx}{1.5}
		\pgfmathsetmacro{\ffoury}{-0.7}
		
		\pgfmathsetmacro{\dslope}{(\foney+\ftwoy+\fthreey+\ffoury-\eoney-\etwoy-\ethreey)/(\fonex+\ftwox+\fthreex+\ffourx-\eonex-\etwox-\ethreex)}
				
		\coordinate (left) at (0, 0);
		\coordinate (p1) at (\eonex, \eoney);
		\coordinate (p2) at (\eonex+\etwox, \eoney+\etwoy);
		\coordinate (right) at (\eonex+\etwox+\ethreex, \eoney+\etwoy+\ethreey);

		\coordinate (q1) at (\fonex, \foney);
		\coordinate (q2) at (\fonex+\ftwox, \foney+\ftwoy);
		\coordinate (q3) at (\fonex+\ftwox+\fthreex, \foney+\ftwoy+\fthreey);
		\coordinate (q4) at (\fonex+\ftwox+\fthreex+\ffourx, \foney+\ftwoy+\fthreey+\ffoury);

		\draw[step=1cm,thick, color=mynicegreen] (left) -- (p1) --  (p2) -- (right);
		\draw[step=1cm,thick, color=red] (left) -- (q1) --  (q2) -- (q3) -- (q4);
		\draw[step=1cm,thick, color=blue] (q4) -- (right);
		
		\draw [fill] (left) circle [radius=0.05];
		\draw [fill] (right) circle [radius=0.05];		
		
		\draw [fill] (p1) circle [radius=0.05];		
		\draw [fill] (p2) circle [radius=0.05];	

		\draw [fill] (q1) circle [radius=0.05];		
		\draw [fill] (q2) circle [radius=0.05];	
		\draw [fill] (q3) circle [radius=0.05];	
		\draw [fill] (q4) circle [radius=0.05];

		\pgfmathsetmacro{\vdx}{(\fonex+\ftwox+\fthreex+\ffourx+\eonex+\etwox+\ethreex)*0.5}	
		\pgfmathsetmacro{\vdy}{(\foney+\ftwoy+\fthreey+\ffoury+\eoney+\etwoy+\ethreey)*0.5}	

		\path (\vdx, \vdy) ++(0.5, 0.2) node {\color{blue}$\gensubvb$};
		\path (p2) ++(0.0, -0.6) node {\color{mynicegreen}$\genextvb$};
		\path (q1) ++(-0.5, 0.2) node {\color{red}$\genquotvb$};
\end{tikzpicture}
\hspace{0.3cm}
\begin{tikzpicture}[scale=0.4]
        \pgfmathsetmacro{\textycoordinate}{5}
		\draw[->, line width=0.6pt] (0, \textycoordinate) -- (1.5,\textycoordinate);
		\draw (0,0) circle [radius=0.00];	
\end{tikzpicture}
\hspace{0.3cm}
\begin{tikzpicture}[scale=0.5]		
		\pgfmathsetmacro{\eonex}{2}
		\pgfmathsetmacro{\eoney}{3}
		\pgfmathsetmacro{\etwox}{3.5}
		\pgfmathsetmacro{\etwoy}{1}
		\pgfmathsetmacro{\ethreex}{4.5}
		\pgfmathsetmacro{\ethreey}{-2.5}
		
		\pgfmathsetmacro{\fonex}{1.5}		
		\pgfmathsetmacro{\foney}{5}		
		\pgfmathsetmacro{\ftwox}{2.5}
		\pgfmathsetmacro{\ftwoy}{2.5}
		\pgfmathsetmacro{\fthreex}{2}
		\pgfmathsetmacro{\fthreey}{0}
		\pgfmathsetmacro{\ffourx}{1.5}
		\pgfmathsetmacro{\ffoury}{-0.7}
		
		\pgfmathsetmacro{\dslope}{(\foney+\ftwoy+\fthreey+\ffoury-\eoney-\etwoy-\ethreey)/(\fonex+\ftwox+\fthreex+\ffourx-\eonex-\etwox-\ethreex)}

		\pgfmathsetmacro{\ptwoy}{\foney+\dslope*(\eonex-\fonex)}	
		\pgfmathsetmacro{\pthreey}{\ptwoy+\ftwoy+\dslope*(\etwox-\ftwox)}
				
		\coordinate (left) at (0, 0);
		\coordinate (p1) at (\eonex, \eoney);
		\coordinate (p2) at (\eonex+\etwox, \eoney+\etwoy);
		\coordinate (right) at (\eonex+\etwox+\ethreex, \eoney+\etwoy+\ethreey);

		\coordinate (q1) at (\fonex, \foney);
		\coordinate (q2) at (\fonex+\ftwox, \foney+\ftwoy);
		\coordinate (q3) at (\fonex+\ftwox+\fthreex, \foney+\ftwoy+\fthreey);
		\coordinate (q4) at (\fonex+\ftwox+\fthreex+\ffourx, \foney+\ftwoy+\fthreey+\ffoury);
		
		\coordinate (s1) at (q1);
		\coordinate (s2) at (\eonex, \ptwoy);
		\coordinate (s3) at (\eonex+\ftwox, \ptwoy+\ftwoy);
		\coordinate (s4) at (\eonex+\etwox, \pthreey);
		\coordinate (s5) at (\eonex+\etwox+\fthreex, \pthreey+\fthreey);
		\coordinate (s6) at (\eonex+\etwox+\fthreex+\ffourx, \pthreey+\fthreey+\ffoury);
		
		\draw[step=1cm,thick, color=mynicegreen] (left) -- (p1) --  (p2) -- (right);		
		\draw[step=1cm,thick, color=red] (left) -- (q1);
		\draw[step=1cm,thick,dotted, color=red] (q1) -- (q2) -- (q3) -- (q4);
		\draw[step=1cm,thick,dotted, color=blue] (q4) -- (s6);
		\draw[step=1cm,thick, color=blue] (s1) -- (s2);
		\draw[step=1cm,thick, color=red] (s2) -- (s3);
		\draw[step=1cm,thick, color=blue] (s3) -- (s4);
		\draw[step=1cm,thick, color=red] (s4) -- (s5) -- (s6);
		\draw[step=1cm,thick, color=blue] (s6) -- (right);
		
		\draw [fill] (left) circle [radius=0.05];
		\draw [fill] (right) circle [radius=0.05];		
		
		\draw [fill] (p1) circle [radius=0.05];		
		\draw [fill] (p2) circle [radius=0.05];	

		\draw [fill] (q1) circle [radius=0.05];		
		\draw [fill] (q2) circle [radius=0.05];	
		\draw [fill] (q3) circle [radius=0.05];	
		\draw [fill] (q4) circle [radius=0.05];	
				
		\draw [fill] (s1) circle [radius=0.05];		
		\draw [fill] (s2) circle [radius=0.05];		
		\draw [fill] (s3) circle [radius=0.05];		
		\draw [fill] (s4) circle [radius=0.05];			
		\draw [fill] (s5) circle [radius=0.05];		
		\draw [fill] (s6) circle [radius=0.05];	
		
		\path (s3) ++(1.1, -1.0) node {\color{purple}$\genpolygon$};
\end{tikzpicture}
\caption{Illustration of the conditions in Theorem \ref{classification of extensions, semistable case intro}}
\end{figure}

\begin{theorem}\label{classification of extensions, general case intro}
Let $\gensubvb$, $\genextvb$ and $\genquotvb$ be vector bundles on $\schff_{\genlocfield, \FFfield}$. 
There exists a short exact sequence 
\[
0 \longrightarrow \gensubvb \longrightarrow \genextvb \longrightarrow \genquotvb \longrightarrow 0
\]
if and only if 
$\genextvb$ admits a filtration
\[ \gensubvb = \genextvb_0 \subset \genextvb_1 \subset \cdots \subset \genextvb_r = \genextvb\]
such that the induced filtration
\[ 0 = \genextvb_0/\gensubvb \subset \genextvb_1/\gensubvb \subset \cdots \subset \genextvb_r/\gensubvb = \genextvb/\gensubvb\]
coincides with the Harder-Narasimhan filtration of $\genquotvb$.
\end{theorem}

Let us make some remarks about our main results. By Theorem \ref{classification of extensions, semistable case intro}, we can formulate Theorem \ref{classification of extensions, general case intro} 
purely in terms of Harder-Narasimhan polygons; 
we refer the readers to Theorem \ref{classification of extensions, general case} for a precise statement. 
After completing this work, we became aware that Chen-Tong \cite{CT_weakaddandnewton} independently obtained our main results in a similar way. 
For vector bundles on $\PP^1$, there is an analogue of our main results due to Schlesinger \cite{Schlesinger_extvbonP1}.

\subsection*{Acknowledgments} 
The author would like to thank the anonymous referee of the article \cite{Hong_extvbss} for a valuable feedback which eventually led to the discovery of Theorem \ref{classification of extensions, semistable case intro}. The author is also grateful for various suggestions from the anonymous referee of this article. 


\section{Vector bundles on the Fargues-Fontaine curve}\label{background}

Throughout this paper, we fix a field $\FFfield$ of characteristic $p>0$ which is complete, nonarchimedean, and algebraically closed. We also let $\genlocfield$ denote an arbitrary nonarchimedean local field whose residue field is finite of characteristic $p$. 




\begin{defn}\label{adicFFC} Let $\integerring{\genlocfield}$ and $\integerring{\FFfield}$ respectively denote the valuation rings of $\genlocfield$ and $\FFfield$. Fix a uniformizer $\uniformizer$ of $\genlocfield$ and a pseudouniformizer $\pseudounif$ of $\FFfield$. Let $q$ be the number of elements in the residue field of $\genlocfield$. 

\begin{enumerate}
\item If $\genlocfield$ is of equal characteristic, we set
\[ \Ycal_{\genlocfield, \FFfield} := \Spa(\integerring{\FFfield}[\![\uniformizer]\!])\setminus\{|\uniformizer \pseudounif|=0\},\]
and define the \emph{adic Fargues-Fontaine curve} associated to the pair $(\genlocfield, \FFfield)$ by
\[ \adicff_{\genlocfield, \FFfield}:= \Ycal_{\genlocfield, \FFfield}/\genfrob^\Z\]
where $\genfrob$ denotes the automorphism of $\Ycal_{\genlocfield, \FFfield}$ induced by the $q$-Frobenius automorphism on $\integerring{\FFfield}[\![\uniformizer]\!]$.
\smallskip

\item\label{definition of adic ff} If $\genlocfield$ is of mixed characteristic, we set
\[ \Ycal_{\genlocfield, \FFfield} := \Spa(\witt_{\integerring{\genlocfield}}(\integerring{\FFfield})\setminus\{|\uniformizer [\pseudounif]|=0\},\]
where $\witt_{\integerring{\genlocfield}}(\integerring{\FFfield})$ denotes the ring of ramified Witt vectors over $\integerring{\FFfield}$ with coefficients in $\integerring{\genlocfield}$ and the Teichmuller lift $[\pseudounif]$ of $\pseudounif$, and define the \emph{adic Fargues-Fontaine curve} associated to the pair $(\genlocfield, \FFfield)$ by
\[ \adicff_{\genlocfield, \FFfield}:= \Ycal_{\genlocfield, \FFfield}/\genfrob^\Z\]
where $\genfrob$ denotes the automorphism of $\Ycal_{\genlocfield, \FFfield}$ induced by the $q$-Frobenius automorphism on $\witt_{\integerring{\genlocfield}}(\integerring{\FFfield})$.
\smallskip

\item We define the \emph{schematic Fargues-Fontaine curve} associated to the pair $(\genlocfield, \FFfield)$ by
\[
\schff_{\genlocfield, \FFfield} := \Proj \left( \bigoplus_{n \geq 0 } H^0(\Ycal_{\genlocfield, \FFfield}, \strsheaf[\Ycal_{\genlocfield, \FFfield}])^{\genfrob = \uniformizer^n} \right).
\]
\end{enumerate}
\end{defn}

\begin{remark}
For curious readers, we address two nontrivial ingredients in the definition of the adic Fargues-Fontaine curve. 
\begin{enumerate}
\item We may identify the ring of ramified Witt vectors as
\[ \witt_{\integerring{\genlocfield}}(\integerring{\FFfield}) = \witt(\integerring{\FFfield}) \otimes_{\witt(\F_q)} \integerring{\genlocfield}\]
where $\witt(\integerring{\FFfield})$ and $\witt(\F_q)$ denote the rings of Witt vectors over $\integerring{\FFfield}$ and $\F_q$. 
\smallskip

\item The expression $\Ycal_{\genlocfield, \FFfield}/\genfrob^\Z$ is valid due to the fact that the action of $\genfrob$ on $\Ycal_{\genlocfield, \FFfield}$ is properly discontinuous. We refer the readers to \cite[Remark 3.1.9]{Kedlaya_arizona} for a proof of this fact. 
\end{enumerate}
\end{remark}

\begin{theorem}[{\cite[Theorems 6.3.12 and 8.7.7]{KL_relpadic1}}]\label{GAGA for FF curve}
There is a natural map of locally ringed spaces
\[
\adicff_{\finext,\algclosedperfdfield} \longrightarrow \schff_{\finext,\algclosedperfdfield}
\]
which induces by pullback an equivalence of the categories of vector bundles. 
\end{theorem}

In light of Theorem \ref{GAGA for FF curve}, which may be regarded as GAGA for the Fargues-Fontaine curve, we will henceforth identify vector bundles on $\adicff_{\genlocfield, \FFfield}$ with vector bundles on $\schff_{\genlocfield, \FFfield}$.

\begin{defn}\label{o-r-over-s}
Let $\lambda = d/r$ be a rational number written in lowest terms with $r>0$. We define the vector bundle $\trivbundle[\genlocfield, \FFfield](\lambda)$ on $\adicff_{\genlocfield, \FFfield}$ (or on $\schff_{\genlocfield, \FFfield}$) by descending along the map $\Ycal_{\genlocfield, \FFfield} \longrightarrow \Ycal_{\genlocfield, \FFfield}/\genfrob^\Z = \adicff_{\genlocfield, \FFfield}$ the trivial bundle $\strsheaf[\Ycal_{\genlocfield, \FFfield}]^{\oplus r}$ equipped with the isomorphism $\genfrob^*\strsheaf[\Ycal_{\genlocfield, \FFfield}]^{\oplus r} \stackrel{\sim}{\longrightarrow} \strsheaf[\Ycal_{\genlocfield, \FFfield}]^{\oplus r}$ represented by the matrix
\[
\left(\begin{array}{@{}c|c@{}}
& \begin{matrix} 1 & & \\ & \ddots & \\ & & 1 \end{matrix} \\ \hline  \uniformizer^{-d} &  
\end{array}\right).
\]
\end{defn}


\begin{prop}[{\cite[Th\'eor\`eme 6.5.2
]{FF_curve}}]\label{why FF curve is a curve} 
The schematic Fargues-Fontaine curve $\schff_{\genlocfield,\FFfield}$ is a Dedekind scheme over $\genlocfield$, with a natural isomorphism from its Picard group $\Pic(\schff_{\genlocfield, \FFfield})$ to $\Z$ which associates each $d \in \Z$ with $\trivbundle[\genlocfield, \FFfield](d)$. 
\end{prop}

\begin{remark}
While Proposition \ref{why FF curve is a curve} suggests that $\schff_{\genlocfield, \FFfield}$ behaves much as algebraic curves do, $\schff_{\genlocfield, \FFfield}$ itself is not an algebraic curve for not being of finite type over the base field $\genlocfield$. Indeed, the residue field at every closed point on $\schff_{\genlocfield, \FFfield}$ is an algebraically closed and complete extension of $\genlocfield$. 
\end{remark}

\begin{defn}
Let $\genvb$ be an arbitrary nonzero vector bundle on $\schff_{\genlocfield, \FFfield}$. 
\begin{enumerate}[label=(\arabic*)]
\item We write $\rk(\genvb)$ for the rank of $\genvb$, and define the \emph{degree} of $\genvb$ to be the integer $\deg(\genvb)$ which corresponds to the isomorphism class of the determinant line bundle $\wedge^{\rk(\genvb)} (\genvb)$ under the natural isomorphism $\Pic(\schff_{\genlocfield, \FFfield}) \cong \Z$ in Proposition \ref{why FF curve is a curve}. 
\smallskip

\item We define the \emph{slope} of $\genvb$ to be
\[\mu(\genvb) := \dfrac{\deg(\genvb)}{\rk(\genvb)}.\]
\end{enumerate}
\end{defn}

\begin{prop}[{\cite[Proposition 5.6.23]{FF_curve}},{\cite[Proposition 4.1.3]{Kedlaya_slopefiltrations}}]\label{basic properties of stable bundles}
Let $\lambda = d/r$ be a rational number written in lowest terms with $r>0$.
\begin{enumerate}[label=(\arabic*)]
\item\label{rank and degree of stable bundles} The vector bundle $\trivbundle[\genlocfield, \FFfield](\lambda)$ has rank $r$, degree $d$, and slope $\lambda = d/r$.
\smallskip

\item\label{dual of stable bundles} The dual of $\trivbundle[\genlocfield, \FFfield](\lambda)$ is isomorphic to $\trivbundle[\genlocfield, \FFfield](-\lambda)$.
\smallskip


\item\label{vanishing hom between stable bundles} $\Hom(\trivbundle[\genlocfield, \FFfield](\lambda), \trivbundle[\genlocfield, \FFfield](\mu))$ is trivial for all rational numbers $\mu < \lambda$.
\smallskip

\item\label{vanishing ext between stable bundles} $\Ext^1(\trivbundle[\genlocfield, \FFfield](\lambda), \trivbundle[\genlocfield, \FFfield](\mu))$ is trivial for all rational numbers $\mu \geq \lambda$. 
\end{enumerate}
\end{prop}

\begin{prop}[{\cite[Proposition 5.6.23]{FF_curve}}]\label{unramified covers of FF curve}
Let $\genlocfield'$ be an unramified finite extension of $\genlocfield$. Denote by $d$ the degree of $\genlocfield'$ over $\genlocfield$. 
\begin{enumerate}[label=(\arabic*)]
\item There exists a canonical isomorphism $\schff_{\genlocfield', \FFfield} \cong \schff_{\genlocfield, \FFfield} \times_{\Spec(\genlocfield)} \Spec(\genlocfield')$.
\smallskip

\item The projection map 
\[ \pi: \schff_{\genlocfield', \FFfield} \cong \schff_{\genlocfield, \FFfield} \times_{\Spec(\genlocfield)} \Spec(\genlocfield') \longrightarrow \schff_{\genlocfield, \FFfield}\]
induces a natural identification
\[ \pi^*\trivbundle[\genlocfield, \FFfield](\lambda) \cong \trivbundle[\genlocfield', \FFfield](d\lambda)^{\oplus m} \quad\quad \text{ for each } \lambda \in \Q\]
with $m = \rk(\trivbundle[\genlocfield, \FFfield](\lambda))/\rk(\trivbundle[\genlocfield', \FFfield](d\lambda))$. 
\end{enumerate}
\end{prop}

\begin{defn}
Let $\genvb$ be a vector bundle on $\schff_{\genlocfield, \FFfield}$. 
\begin{enumerate}[label=(\arabic*)]
\item We say that $\genvb$ is \emph{stable} if we have $\mu(\gensecvb) < \mu(\genvb)$ for all nonzero subbundles $\gensecvb \subsetneq \genvb$.
\smallskip

\item We say that $\genvb$ is \emph{semistable} if we have $\mu(\gensecvb) \leq \mu(\genvb)$ for all nonzero subbundles $\gensecvb \subsetneq \genvb$. 
\end{enumerate}
\end{defn}

\medskip
\begin{theorem}[{\cite[Th\'eor\`eme 8.2.10]{FF_curve}}]\label{existence of HN decomp}
Let $\genvb$ be a vector bundle on $\schff_{\genlocfield, \FFfield}$. 
\begin{enumerate}[label=(\arabic*)]
\item $\genvb$ is stable of slope $\lambda$ if and only if it is isomorphic to $\trivbundle[\genlocfield, \FFfield](\lambda)$. 
\smallskip

\item $\genvb$ is semistable of slope $\lambda$ if and only if it is isomorphic to $\trivbundle[\genlocfield, \FFfield](\lambda)^{\oplus m}$ for some $m$. 
\smallskip

\item $\genvb$ admits a direct sum decomposition
\begin{equation*}
\genvb \simeq \bigoplus_{i=1}^l \trivbundle[\genlocfield, \FFfield](\lambda_i)^{\oplus m_i} \quad\quad \text{ with } \lambda_i \in \Q,
\end{equation*}
where the $\lambda_i$'s are all distinct and uniquely determined up to permutations. 
\end{enumerate}
\end{theorem}

\begin{remark}
Prior to the work of Fargues-Fontaine \cite{FF_curve}, Theorem \ref{existence of HN decomp} had been obtained in a different language by Hartl-Pink \cite[Theorem 11.1]{HP_equalcharFFcurve} and Kedlaya \cite[Theorem 4.5.7]{Kedlaya_slopefiltrations}.
\end{remark}

\begin{defn}\label{def of HN filtration, decomposition and polygon}
Let $\genvb$ be a vector bundle on $\schff_{\genlocfield, \FFfield}$. Fix a direct sum decomposition
\begin{equation}\label{HN decomposition}
\genvb \simeq \bigoplus_{i=1}^l \trivbundle[\genlocfield, \FFfield](\lambda_i)^{\oplus m_i} \quad\quad \text{ with } \lambda_i \in \Q
\end{equation}
given by Theorem \ref{existence of HN decomp}. 
\begin{enumerate}[label = (\arabic*)]

\item We refer to the numbers $\lambda_i$ as the 
\emph{Harder-Narasimhan (HN) slopes} of $\genvb$, or often simply as the 
\emph{slopes} of $\genvb$,
and write $\mumax(\genvb)$ and $\mumin(\genvb)$ respectively for the maximum and minimum HN slopes of $\genvb$. 
\smallskip

\smallskip

\item For every $\mu \in \Q$, we define the direct summands
\[\genvb^{\geq \mu} := \bigoplus_{\lambda_i \geq \mu}\trivbundle[\genlocfield, \FFfield](\lambda_i)^{\oplus m_i} \quad\quad \text{ and } \quad\quad\genvb^{\leq \mu} := \bigoplus_{\lambda_i \leq \mu}\trivbundle[\genlocfield, \FFfield](\lambda_i)^{\oplus m_i},\]
and similarly define $\genvb^{> \mu}$ and $\genvb^{<\mu}$. 
\smallskip

\item We define the \emph{Harder-Narasimhan (HN) polygon} of $\genvb$, denoted by $\HN(\genvb)$, as the upper convex hull of the points $(0, 0)$ and $\big(\rk(\genvb^{\geq \lambda_i}), \deg(\genvb^{\geq \lambda_i})\big)$.
\end{enumerate}
\end{defn}

\begin{remark}
For our work, the choice of the decomposition \eqref{HN decomposition} does not play a significant role. In fact, since $\genvb$ determines $\lambda_i$ and $m_i$, it consequently determines the isomorphism classes of $\genvb^{\geq \mu}$, $\genvb^{\leq \mu}$, $\genvb^{> \mu}$, and $\genvb^{<\mu}$ for each $\mu \in \Q$, and thus the polygon $\HN(\genvb)$. 
\end{remark}

\section{Extensions and permutations of HN polygons}

Our goal in this section is to state and derive the necessary conditions for the existence of a short exact sequence involving three given vector bundles on the Fargues-Fontaine curve. 

\begin{defn}
A \emph{rationally tuplar polygon} is a graph $\genpolygon$ of a continuous function on a closed interval which satisfies the following properties:
\begin{enumerate}[label=(\roman*)]
\item The endpoints of $\genpolygon$ are integer points, with the left endpoint being $(0, 0)$. 
\smallskip

\item For each integer $i$ in the domain, $\genpolygon$ is linear on the interval $[i-1, i]$ with a rational slope denoted by $\HNslope{\genpolygon}{i}$. 
\end{enumerate} 
\end{defn}


\begin{remark}
For every vector bundle $\genvb$ on $\schff_{\genlocfield, \FFfield}$, its Harder-Narasimhan polygon $\HN(\genvb)$ is a rationally tuplar polygon with the following additional properties:
\begin{itemize}
\item All breakpoints of $\HN(\genvb)$ are integer points. 
\smallskip

\item The slopes of $\HN(\genvb)$ are decreasing. 
\end{itemize}
However, in this article we will consider rationally tuplar polygons which do not necessarily enjoy these properties. 
\end{remark}

\begin{lemma}\label{necessary condition for HN polygon of subsheaves and quotients}
Let $\genvb$ and $\gensecvb$ be arbitrary vector bundles on $\schff_{\genlocfield, \FFfield}$.
\begin{enumerate}[label=(\arabic*)]
\item If $\gensecvb$ is a subsheaf of $\genvb$, we have 
\[\HNslope{\HN(\gensecvb)}{i} \leq \HNslope{\HN(\genvb)}{i} \quad\quad \text{ for } i = 1, \cdots, \rk(\gensecvb).\]

\item If $\gensecvb$ is a quotient of $\genvb$, we have 
\[\HNslope{\HN(\gensecvb)}{i} \geq \HNslope{\HN(\genvb)}{i+\rk(\genvb)-\rk(\gensecvb)} \quad\quad \text{ for } i = 1, \cdots, \rk(\gensecvb).\]
\end{enumerate}
\end{lemma}

\begin{proof}
Let us establish the first statement. Suppose for contradiction that $\gensecvb$ is a subsheaf of $\genvb$ with $\HNslope{\HN(\gensecvb)}{i} > \HNslope{\HN(\genvb)}{i}$ for some positive integer $i \leq \rk(\gensecvb)$. Let us fix an injective bundle map $\gensecvb \inj \genvb$ and write $\mu:= \HNslope{\HN(\gensecvb)}{i}$. We have direct sum decompositions
\[ \genvb \simeq \genvb^{\geq \mu} \oplus \genvb^{< \mu} \quad\quad \text{ and } \quad\quad \gensecvb \simeq \gensecvb^{\geq \mu} \oplus \gensecvb^{< \mu}.\]
Since every bundle map from $\gensecvb^{\geq \mu}$ to $\genvb^{< \mu}$ is zero by Proposition \ref{basic properties of stable bundles}, the injective map $\gensecvb \inj \genvb$ restricts to an injective map $\gensecvb^{\geq \mu} \inj \genvb^{\geq \mu}$. Hence we find $\rk(\gensecvb^{\geq \mu}) \leq \rk(\genvb^{\geq \mu})$. On the other hand, we have $\rk(\gensecvb^{\geq \mu}) \geq i > \rk(\genvb^{\geq \mu})$ by the concavity of HN polygons and the inequality $\HNslope{\HN(\gensecvb)}{i} > \HNslope{\HN(\genvb)}{i}$. 
We thus obtain a contradiction as desired.

For the second statement, we now assume that $\gensecvb$ is a quotient of $\genvb$. Let us write $\genvb^\vee$ and $\gensecvb^\vee$ respectively for the duals of $\genvb$ and $\gensecvb$. Since $\gensecvb^\vee$ is a subbundle of $\genvb^\vee$, we use Proposition \ref{basic properties of stable bundles} and the first statement to find
\[ \HNslope{\HN(\gensecvb)}{i} = - \HNslope{\HN(\gensecvb^\vee)}{\rk(\gensecvb)+1-i} \geq - \HNslope{\HN(\genvb^\vee)}{\rk(\gensecvb)+1-i} = \HNslope{\HN(\genvb)}{i+\rk(\genvb)-\rk(\gensecvb)}\]
for each $i = 1, \cdots, \rk(\gensecvb)$, thereby deducing the second statement. 
\end{proof}

\begin{remark}
The converse of the first statement is true by a previous result of the author \cite[Theorem 1.1.2]{Hong_subvb},
whereas the converse of the second statement does not by another previous result of the author \cite[Theorem 1.1.2]{Hong_quotvb}. 
However, our argument won't need these results.
\end{remark}

\begin{defn}\label{dominance order definition}
Let $\genpolygon$ and $\gensecpolygon$ be rational tuplar polygons. We say that $\genpolygon$ \emph{dominates} $\gensecpolygon$ and write $\genpolygon \geq \gensecpolygon$ if $\genpolygon$ and $\gensecpolygon$ satisfy the following properties:
\begin{enumerate}[label=(\roman*)]
\item $\genpolygon$ and $\gensecpolygon$ have the same endpoints. 
\smallskip

\item If $r$ denotes the $x$-coordinate of their common right endpoint, we have
\[ \sum_{i = 1}^j \HNslope{\genpolygon}{i} \geq \sum_{i=1}^j \HNslope{\gensecpolygon}{i} \quad\quad \text{ for each } j = 1, \cdots, r.\]
\end{enumerate}
\end{defn}

\begin{remark}
In other words, we have $\genpolygon \geq \gensecpolygon$ if and only if $\genpolygon$ lies on or above $\gensecpolygon$ with the same endpoints, 
as in Mazur's inequality. 
\end{remark}

\begin{lemma}\label{dominance order is a partial order}
The binary relation $\geq$ is a partial order on the set of rational tuplar polygons. 
\end{lemma}

\begin{proof}
This is straightforward to check using Definition \ref{dominance order definition}. 
\end{proof}

\begin{defn}\label{suitable permutation of HN polygon for extension}
Given vector bundles $\gensubvb$, $\genextvb$ and $\genquotvb$ on $\schff_{\genlocfield, \FFfield}$, we define 
a \emph{$(\gensubvb, \genextvb, \genquotvb)$-permutation} of $\HN(\gensubvb \oplus \genquotvb)$ to be a rationally tuplar polygon $\genpolygon \geq \HN(\genextvb)$ with the following properties:
\begin{enumerate}[label=(\roman*)]
\item The tuple $(\HNslope{\genpolygon}{i})$ is a permutation of the tuple $(\HNslope{\HN(\gensubvb \oplus \genquotvb)}{i})$. 
\smallskip


\item For each $i = 1, \cdots, \rk(\genextvb)$, we have
\smallskip
\begin{itemize}
\item $\HNslope{\genpolygon}{i} < \HNslope{\HN(\genextvb)}{i}$ only if $\HNslope{\genpolygon}{i}$ occurs as a slope of $\gensubvb$, and
\smallskip

\item $\HNslope{\genpolygon}{i} > \HNslope{\HN(\genextvb)}{i}$ only if $\HNslope{\genpolygon}{i}$ occurs as a slope of $\genquotvb$.
\end{itemize}
\end{enumerate}
\end{defn}

\begin{figure}[H]
\begin{tikzpicture}[scale=0.5]	
		\coordinate (left) at (0, 0);
		\coordinate (p1) at (2, 3);
		\coordinate (p2) at (6, 4);
		\coordinate (right) at (10, 2);

		\coordinate (q1) at (1, 4);
		\coordinate (q2) at (3, 6);
		\coordinate (q3) at (5, 7);
		\coordinate (q4) at (7, 7);
		\coordinate (q5) at (9, 5);
				
		\draw[step=1cm,thick, color=mynicegreen] (left) -- (p1) --  (p2) -- (right);
		\draw[step=1cm,thick, color=red] (left) -- (q1) -- (q2) -- (q3);
		\draw[step=1cm,thick, color=blue] (q3) -- (q4) -- (q5) -- (right);

		\draw [fill] (q1) circle [radius=0.05];		
		\draw [fill] (q2) circle [radius=0.05];		
		\draw [fill] (q3) circle [radius=0.05];		
		\draw [fill] (q4) circle [radius=0.05];		
		\draw [fill] (q5) circle [radius=0.05];		
		\draw [fill] (left) circle [radius=0.05];
		\draw [fill] (right) circle [radius=0.05];		
		
		\draw [fill] (p1) circle [radius=0.05];		
		\draw [fill] (p2) circle [radius=0.05];		
		
		\path (q5) ++(0.2, 0.6) node {\color{blue}$\gensubvb$};
		\path (p2) ++(0.2, 0.6) node {\color{mynicegreen}$\genextvb$};
		\path (q1) ++(-0.5, 0.2) node {\color{red}$\genquotvb$};
\end{tikzpicture}
\hspace{0.3cm}
\begin{tikzpicture}[scale=0.4]
        \pgfmathsetmacro{\textycoordinate}{5}
		\draw[->, line width=0.6pt] (0, \textycoordinate) -- (1.5,\textycoordinate);
		\draw (0,0) circle [radius=0.00];	
\end{tikzpicture}
\hspace{0.3cm}
\begin{tikzpicture}[scale=0.5]
		\coordinate (left) at (0, 0);
		\coordinate (p1) at (2, 3);
		\coordinate (p2) at (6, 4);
		\coordinate (right) at (10, 2);

		\coordinate (q1) at (1, 4);
		\coordinate (q2) at (3, 6);
		\coordinate (q3) at (5, 7);
		\coordinate (q4) at (7, 7);
		\coordinate (q5) at (9, 5);
		
		\coordinate (r1) at (q1);
		\coordinate (r2) at (2, 4);
		\coordinate (r3) at (4, 6);
		\coordinate (r4) at (5, 6);
		\coordinate (r5) at (q4);
		\coordinate (r6) at (q5);
				
		\draw[step=1cm,thick, color=mynicegreen] (left) -- (p1) --  (p2) -- (right);
		\draw[step=1cm,thick, color=red] (left) -- (q1);
		\draw[step=1cm,thick,dotted, color=red] (q1) -- (q2) -- (q3);
		\draw[step=1cm,thick,dotted, color=blue] (q3) -- (q4);
		\draw[step=1cm,thick, color=blue] (q4) -- (q5) -- (right);
		
		\draw[step=1cm,thick, color=blue] (r1) -- (r2);
		\draw[step=1cm,thick, color=red] (r2) -- (r3);
		\draw[step=1cm,thick, color=blue] (r3) -- (r4);
		\draw[step=1cm,thick, color=red] (r4) -- (r5);

		\draw [fill] (q1) circle [radius=0.05];		
		\draw [fill] (q2) circle [radius=0.05];		
		\draw [fill] (q3) circle [radius=0.05];		
		\draw [fill] (q4) circle [radius=0.05];		
		\draw [fill] (q5) circle [radius=0.05];		
		\draw [fill] (left) circle [radius=0.05];
		\draw [fill] (right) circle [radius=0.05];		
		
		\draw [fill] (p1) circle [radius=0.05];		
		\draw [fill] (p2) circle [radius=0.05];	

		\draw [fill] (r1) circle [radius=0.05];		
		\draw [fill] (r2) circle [radius=0.05];		
		\draw [fill] (r3) circle [radius=0.05];		
		\draw [fill] (r4) circle [radius=0.05];		
		\draw [fill] (r5) circle [radius=0.05];		
		\draw [fill] (r6) circle [radius=0.05];		
		
\end{tikzpicture}
\caption{Illustration of the conditions in Definition \ref{suitable permutation of HN polygon for extension}}
\end{figure}

\begin{lemma}\label{partition lemma for extension-permutations}
Let $\gensubvb$, $\genextvb$ and $\genquotvb$ be arbitrary vector bundles on $\schff_{\genlocfield, \FFfield}$. 
A rationally tuplar polygon $\genpolygon \geq \HN(\genextvb)$ is a $(\gensubvb, \genextvb, \genquotvb)$-permutation of $\HN(\gensubvb \oplus \genquotvb)$ if and only if there exists 
an ordered pair $(S_\gensubvb, S_\genquotvb)$ of sets satisfying the following properties:
\begin{enumerate}[label=(\roman*)]
\item The sets $S_\gensubvb$ and $S_\genquotvb$ form a partition of the index set $\{1, \cdots, \rk(\genextvb)\}$. 
\smallskip

\item The tuple $(\HNslope{\genpolygon}{i})_{i \in S_\gensubvb}$ permutes the tuple $(\HNslope{\HN(\gensubvb)}{i})$ with
\[  \HNslope{\genpolygon}{i} \leq \HNslope{\HN(\genextvb)}{i} \text{ for all } i \in S_\gensubvb.\]

\item The tuple $(\HNslope{\genpolygon}{i})_{i \in S_\genquotvb}$ permutes the tuple $(\HNslope{\HN(\genquotvb)}{i})$ with 
\[  \HNslope{\genpolygon}{j} \geq \HNslope{\HN(\genextvb)}{j} \text{ for all } j \in S_\genquotvb.\]
\end{enumerate}
\end{lemma}

\begin{proof}
This is evident by Definition \ref{suitable permutation of HN polygon for extension}. 
\end{proof}

\begin{remark}
The inequalities in Lemma \ref{partition lemma for extension-permutations} are not strict, whereas the inequalities in Definition \ref{suitable permutation of HN polygon for extension} are strict. In fact, the inequalities in Lemma \ref{partition lemma for extension-permutations} are essentially contrapositives of the inequalities in Definition \ref{suitable permutation of HN polygon for extension}. 
\end{remark}

\begin{defn}\label{sorted P-partition pair def}
Let $\gensubvb$, $\genextvb$ and $\genquotvb$ be vector bundles on $\schff_{\genlocfield, \FFfield}$ with
a $(\gensubvb, \genextvb, \genquotvb)$-permutation $\genpolygon$ of $\HN(\gensubvb \oplus \genquotvb)$.
\begin{enumerate}[label=(\arabic*)] 
\item We refer to an ordered pair $(S_\gensubvb, S_\genquotvb)$ as in Lemma \ref{partition lemma for extension-permutations} as a \emph{$\genpolygon$-partition pair}. 
\smallskip

\item Given a $\genpolygon$-partition pair $(S_\gensubvb, S_\genquotvb)$, we say that $\genpolygon$ is \emph{$(S_\gensubvb, S_\genquotvb)$-sorted} if we have 
\[ (\HNslope{\genpolygon}{i})_{i \in S_\gensubvb} = (\HNslope{\HN(\gensubvb)}{i}) \quad \text{ and } \quad (\HNslope{\genpolygon}{j})_{j \in S_\genquotvb} = (\HNslope{\HN(\genquotvb)}{j}).\]

\end{enumerate}
\end{defn}

\begin{remark}
For our work, we will consider tuples of various sizes and/or index sets, as in the two equalities above. When we write an equality of tuples, we always specify the index set for one side to specify the size of tuples. In addition, when we do not specify the index set for a tuple, it means that the index set is take to be maximal (ranging from 1 to the length of the tuple).  
\end{remark}

\begin{lemma}\label{sorted suitable permutation}
Let $\gensubvb$, $\genextvb$ and $\genquotvb$ be vector bundles on $\schff_{\genlocfield, \FFfield}$. Suppose that there exists a $(\gensubvb, \genextvb, \genquotvb)$-permutation $\genpolygon$ of $\HN(\gensubvb \oplus \genquotvb)$ with a $\genpolygon$-partition pair $(S_\gensubvb, S_\genquotvb)$. There exists a $(\gensubvb, \genextvb, \genquotvb)$-permutation $\gensecpolygon$ of $\HN(\gensubvb \oplus \genquotvb)$ with the following properties:
\begin{enumerate}[label=(\roman*)]
\item\label{property of having a fixed partition pair} $(S_\gensubvb, S_\genquotvb)$ is a $\gensecpolygon$-partition pair.
\smallskip

\item\label{property of being sorted} $\gensecpolygon$ is $(S_\gensubvb, S_\genquotvb)$-sorted. 
\end{enumerate}
\end{lemma}

\begin{proof}
Since there are finitely many $(\gensubvb, \genextvb, \genquotvb)$-permutations of $\HN(\gensubvb \oplus \genquotvb)$, we can take a permutation $\gensecpolygon$ of $\HN(\gensubvb \oplus \genquotvb)$ which is maximal among those with property \ref{property of having a fixed partition pair}. We wish to show that $\gensecpolygon$ is $(S_\gensubvb, S_\genquotvb)$-sorted. By concavity of HN polygons, it suffices to prove that the tuples $(\HNslope{\gensecpolygon}{i})_{i \in S_\gensubvb}$ and $(\HNslope{\gensecpolygon}{j})_{i \in S_\genquotvb}$ are sorted in descending order. 

Let us first verify that $(\HNslope{\gensecpolygon}{i})_{i \in S_\gensubvb}$ is sorted in descending order. Suppose for contradiction that there exist integers $a, b \in S_\gensubvb$ with $a<b$ and $\HNslope{\gensecpolygon}{a} < \HNslope{\gensecpolygon}{b}$. We have
\begin{equation}\label{slope inequality for unsorted D} 
\HNslope{\HN(\genextvb)}{b} \leq \HNslope{\HN(\genextvb)}{a} \leq \HNslope{\gensecpolygon}{a} < \HNslope{\gensecpolygon}{b}
\end{equation}
where the first inequality follows from the concavity of $\HN(\genextvb)$. Take $\gensecpolygon'$ to be the rationally tuplar polygon such that the tuple $(\HNslope{\gensecpolygon'}{i})$ swaps the positions of $\HNslope{\gensecpolygon}{a}$ and $\HNslope{\gensecpolygon}{b}$ in the tuple $(\HNslope{\gensecpolygon}{i})$. It follows from the inequality \eqref{slope inequality for unsorted D} that $\gensecpolygon'$ is a $(\gensubvb, \genextvb, \genquotvb)$-permutation of $\HN(\gensubvb \oplus \genquotvb)$ with $\gensecpolygon' \geq \gensecpolygon$ and $\gensecpolygon' \neq \gensecpolygon$ such that $(S_\gensubvb, S_\genquotvb)$ is a $\gensecpolygon''$-partition pair. Hence we obtain a contradiction to the maximality of $\gensecpolygon$ as desired. 

It remains to check that $(\HNslope{\gensecpolygon}{j})_{j \in S_\genquotvb}$ is sorted in descending order. Suppose for contradiction that there exist integers $c, d \in S_\genquotvb$ with $c<d$ and $\HNslope{\gensecpolygon}{c} < \HNslope{\gensecpolygon}{d}$. We have
\begin{equation}\label{slope inequality for unsorted F} 
\HNslope{\gensecpolygon}{c} < \HNslope{\gensecpolygon}{d} \leq \HNslope{\HN(\genextvb)}{d} \leq \HNslope{\HN(\genextvb)}{c}
\end{equation}
where the last inequality follows from the concavity of $\HN(\genextvb)$. Take $\gensecpolygon''$ to be the rationally tuplar polygon such that the tuple $(\HNslope{\gensecpolygon''}{j})$ swaps the positions of $\HNslope{\gensecpolygon}{c}$ and $\HNslope{\gensecpolygon}{d}$ in the tuple $(\HNslope{\gensecpolygon}{j})$. It follows from the inequality \eqref{slope inequality for unsorted F} that $\gensecpolygon''$ is a $(\gensubvb, \genextvb, \genquotvb)$-permutation of $\HN(\gensubvb \oplus \genquotvb)$ with $\gensecpolygon'' \geq \gensecpolygon$ and $\gensecpolygon'' \neq \gensecpolygon$ such that $(S_\gensubvb, S_\genquotvb)$ is a $\gensecpolygon''$-partition pair. Hence we obtain a contradiction to the maximality of $\gensecpolygon$ as desired. 
\end{proof}

\begin{defn}
Let $\gensubvb$, $\genextvb$ and $\genquotvb$ be vector bundles on $\schff_{\genlocfield, \FFfield}$ with a short exact sequence
\[ 0 \longrightarrow \gensubvb \longrightarrow \genextvb \longrightarrow \genquotvb \longrightarrow 0.\]
Let $A$ and $B$ be sets which form a partition of the set $\{1, \cdots, \rk(\genextvb)\}$. An \emph{$(A, B)$-decomposition} of $\genextvb$ is a direct sum decomposition $\genextvb \simeq \genextvb_A \oplus \genextvb_B$ satisfying the following properties:
\begin{enumerate}[label=(\roman*)]
\item We have $(\HNslope{\HN(\genextvb_A)}{i}) = (\HNslope{\HN(\genextvb)}{i}))_{i \in A}$ and $(\HNslope{\HN(\genextvb_B)}{i}) = (\HNslope{\HN(\genextvb)}{i}))_{i \in B}$.
\smallskip

\item The map $\genextvb_B \inj \genextvb \surj \genquotvb$ is injective. 
\end{enumerate}
\end{defn}

\begin{prop}\label{necessary conditions for HN polygon of extensions with integer slopes}
Let $\gensubvb$, $\genextvb$ and $\genquotvb$ be vector bundles on $\schff_{\genlocfield, \FFfield}$ with integer slopes. If there exists a short exact sequence 
\begin{equation}\label{short exact sequence, necessity part integer slopes} 
0 \longrightarrow \gensubvb \longrightarrow \genextvb \longrightarrow \genquotvb \longrightarrow 0,
\end{equation}
then there exists a $(\gensubvb, \genextvb, \genquotvb)$-permutation $\genpolygon$ of $\HN(\gensubvb \oplus \genquotvb)$ with a $\genpolygon$-partition pair $(S_\gensubvb, S_\genquotvb)$ and an $(S_\gensubvb, S_\genquotvb)$-decomposition of $\genextvb$. 
\end{prop}

\begin{proof}
%

Let us proceed by induction on $\rk(\genextvb)$. The assertion is trivial when $\genextvb$ is zero. We henceforth assume that $\genextvb$ is nonzero. 
The vector bundle $\genextvb$ admits a decomposition 
\[ \genextvb \simeq \genred{\genextvb} \oplus \trivbundle[\genlocfield, \FFfield](\mumin(\genextvb))\]
where $\genred{\genextvb}$ is a vector bundle on $\schff_{\genlocfield, \FFfield}$ with $\rk(\genred{\genextvb}) = \rk(\genextvb) - 1$; indeed, $\HN(\genred{\genextvb})$ is given by the restriction of $\HN(\genvb)$ on the interval $[0, \rk(\genextvb)-1]$. We write $\genred{\gensubvb}$ for the preimage of $\genred{\genextvb}$ under the map $\gensubvb \longrightarrow \genextvb$, and $\genred{\genquotvb}$ for the image of $\genred{\genextvb}$ under the map $\genextvb \longrightarrow \genquotvb$. Then we have a commutative diagram of short exact sequences
\begin{equation}\label{induction diagram of short exact sequences, necessity part} 
\begin{tikzcd}
0 \arrow[r]& \genred{\gensubvb} \arrow[r]\arrow[d, hookrightarrow]& \genred{\genextvb} \arrow[r]\arrow[d, hookrightarrow]& \genred{\genquotvb} \arrow[r]\arrow[d, hookrightarrow]& 0\\
0 \arrow[r]& \gensubvb \arrow[r]& \genextvb \arrow[r]& \genquotvb \arrow[r]& 0
\end{tikzcd}
\end{equation}
where all vertical maps are injective. By the induction hypothesis, there exists 
a $(\genred{\gensubvb}, \genred{\genextvb}, \genred{\genquotvb})$-permutation $\genred{\genpolygon}$ of $\HN(\genred{\gensubvb} \oplus \genred{\genquotvb})$ with an $\genred{\genpolygon}$-partition pair $(S_{\genred{\gensubvb}}, S_{\genred{\genquotvb}})$ and an $(S_{\genred{\gensubvb}}, S_{\genred{\genquotvb}})$-decomposition $\genred{\genextvb} \simeq \genred{\genextvb}_{\genred{\gensubvb}} \oplus \genred{\genextvb}_{\genred{\genquotvb}}$. In light of Lemma \ref{sorted suitable permutation}, we may assume that $\genred{\genpolygon}$ is $(S_{\genred{\gensubvb}}, S_{\genred{\genquotvb}})$-sorted. In addition, we find that $\rk(\genred{\genquotvb})$ is equal to either $\rk(\genquotvb) -1$ or $\rk(\genquotvb)$.

We consider the case where $\rk(\genred{\genquotvb})$ is equal to $\rk(\genquotvb) -1$. The surjective map $\genextvb \surj \genquotvb$ induces a surjective map $\trivbundle[\genlocfield, \FFfield](\mumin(\genextvb)) \simeq \genextvb/\genred{\genextvb} \surj \genquotvb/\genred{\genquotvb}$, which must be an isomorphism as both $\trivbundle[\genlocfield, \FFfield](\mumin(\genextvb))$ and $\genquotvb/\genred{\genquotvb}$ are of rank $1$. This isomorphism yields a map
\[ \genquotvb/\genred{\genquotvb} \simeq \trivbundle[\genlocfield, \FFfield](\mumin(\genextvb)) \inj \genextvb \surj \genquotvb\]
which splits the short exact sequence
\[ 0 \longrightarrow \genred{\genquotvb} \longrightarrow \genquotvb \longrightarrow \genquotvb/\genred{\genquotvb} \longrightarrow 0\]
and consequently induces a direct sum decomposition
\[ \genquotvb \simeq \genred{\genquotvb} \oplus \genquotvb/\genred{\genquotvb} \simeq \genred{\genquotvb} \oplus \trivbundle[\genlocfield, \FFfield](\mumin(\genextvb)).\]
In particular, we find $\mumin(\genquotvb) \leq \mumin(\genextvb)$. Meanwhile, Lemma \ref{necessary condition for HN polygon of subsheaves and quotients} implies $\mumin(\genquotvb) \geq \mumin(\genextvb)$ as $\genquotvb$ is a quotient of $\genextvb$. Hence we have $\mumin(\genextvb) = \mumin(\genquotvb)$. 
We also obtain an isomorphism $\genred{\gensubvb} \simeq \gensubvb$ by applying the snake lemma to the diagram \eqref{induction diagram of short exact sequences, necessity part}. 
Let us now take $\genpolygon$ to be the rationally tuplar polygon with
\[ \HNslope{\genpolygon}{i} = \begin{cases} \HNslope{\genred{\genpolygon}}{i} & \text{ for } i < \rk(\genextvb), \\  \mumin(\genquotvb) & \text{ for } i = \rk(\genextvb) \end{cases}\]
We also set $S_\gensubvb:= S_{\genred{\gensubvb}}$, $S_\genquotvb:= S_{\genred{\genquotvb}} \cup \{\rk(\genextvb)\}$, $\genextvb_\gensubvb:= \genred{\genextvb}_{\genred{\gensubvb}}$, and $\genextvb_\genquotvb := \genred{\genextvb}_{\genred{\genquotvb}} \oplus \trivbundle[\genlocfield, \FFfield](\mumin(\genextvb))$. It is then straightforward to verify that $\genpolygon$ is a $(\gensubvb, \genextvb, \genquotvb)$-permutation of $\HN(\gensubvb \oplus \genquotvb)$ with a $\genpolygon$-partition pair $(S_\gensubvb, S_\genquotvb)$ and an $(S_\gensubvb, S_\genquotvb)$-decomposition $\genextvb \simeq \genextvb_\gensubvb \oplus \genextvb_\genquotvb$,
as illustrated by Figure \ref{suitable permutation in the split case}. 
\begin{figure}[H]
\begin{tikzpicture}[scale=0.5]	
		\coordinate (left) at (0, 0);
		\coordinate (p1) at (3, 3);
		\coordinate (p2) at (8, 4);
		\coordinate (right) at (10, 2);

		\coordinate (q1) at (2, 4);
		\coordinate (q2) at (5, 5);
		\coordinate (q3) at (6, 4);
		\coordinate (q4) at (7, 5);
		\coordinate (q5) at (9, 4);
				
		\draw[step=1cm,thick, color=mynicegreen] (left) -- (p1) --  (p2) -- (right);
		\draw[step=1cm,thick, color=red] (left) -- (q1) -- (q2);
		\draw[step=1cm,thick, color=orange] (q2) -- (q3);
		\draw[step=1cm,thick, color=blue] (q3) -- (q4) -- (q5) -- (right);

		\draw [fill] (q1) circle [radius=0.05];		
		\draw [fill] (q2) circle [radius=0.05];		
		\draw [fill] (q3) circle [radius=0.05];		
		\draw [fill] (q4) circle [radius=0.05];		
		\draw [fill] (q5) circle [radius=0.05];	
		\draw [fill] (left) circle [radius=0.05];
		\draw [fill] (right) circle [radius=0.05];		
		
		\draw [fill] (p1) circle [radius=0.05];		
		\draw [fill] (p2) circle [radius=0.05];		
		
		\path (q5) ++(1.4, 0.2) node {\color{blue}$\genred{\gensubvb} = \gensubvb$};
		\path (p1) ++(0.2, -0.6) node {\color{mynicegreen}$\genextvb$};
		\path (q1) ++(-0.5, 0.2) node {\color{red}$\genred{\genquotvb}$};
		\path (q2) ++(1, 0.2) node {\color{orange}$\genquotvb/\genred{\genquotvb}$};
\end{tikzpicture}
\hspace{0.3cm}
\begin{tikzpicture}[scale=0.4]
        \pgfmathsetmacro{\textycoordinate}{4}
		\draw[->, line width=0.6pt] (0, \textycoordinate) -- (1.5,\textycoordinate);
		\draw (0,0) circle [radius=0.00];	
\end{tikzpicture}
\hspace{0.3cm}
\begin{tikzpicture}[scale=0.5]
		\coordinate (left) at (0, 0);
		\coordinate (p1) at (3, 3);
		\coordinate (p2) at (8, 4);
		\coordinate (right) at (10, 2);

		\coordinate (q1) at (2, 4);
		\coordinate (q2) at (5, 5);
		\coordinate (q3) at (6, 4);
		\coordinate (q4) at (7, 5);
		\coordinate (q5) at (9, 4);
		
		\coordinate (r1) at (q1);
		\coordinate (r2) at (3, 5);
		\coordinate (r3) at (6, 6);
		\coordinate (r4) at (8, 5);
		\coordinate (r5) at (9, 3);
				
		\draw[step=1cm,thick, color=mynicegreen] (left) -- (p1) --  (p2) -- (r5);
		\draw[step=1cm,thick, color=red] (left) -- (q1);
		
		\draw[step=1cm,thick, color=blue] (r1) -- (r2);
		\draw[step=1cm,thick, color=red] (r2) -- (r3);
		\draw[step=1cm,thick, color=blue] (r3) -- (r4) -- (r5);
		\draw[step=1cm,thick, color=orange] (r5) -- (right);
		
		\draw [fill] (left) circle [radius=0.05];
		\draw [fill] (right) circle [radius=0.05];		
		
		\draw [fill] (p1) circle [radius=0.05];		
		\draw [fill] (p2) circle [radius=0.05];	

		\draw [fill] (r1) circle [radius=0.05];		
		\draw [fill] (r2) circle [radius=0.05];		
		\draw [fill] (r3) circle [radius=0.05];		
		\draw [fill] (r4) circle [radius=0.05];		
		\draw [fill] (r5) circle [radius=0.05];		
		
		\path (p1) ++(0.2, -0.6) node {\color{mynicegreen}$\genred{\genextvb}$};
\end{tikzpicture}
\caption{Construction of $\genpolygon$ in the case $\rk(\genred{\genquotvb}) = \rk(\genquotvb) -1$}\label{suitable permutation in the split case}
\end{figure}

It remains to consider the case where $\rk(\genred{\genquotvb})$ and $\rk(\genquotvb)$ are equal. 
Let us set 
\[S_\gensubvb := S_{\genred{\gensubvb}} \cup \{\rk(\genextvb)\}, \quad S_\genquotvb := S_{\genred{\genquotvb}}, \quad \genextvb_\gensubvb:= \genred{\genextvb}_{\genred{\gensubvb}} \oplus \trivbundle[\genlocfield, \FFfield](\mumin(\genextvb)), \quad \genextvb_\genquotvb:= \genred{\genextvb}_{\genred{\genquotvb}}.\]
We also take $\genpolygon$ to be the rationally tuplar polygon with 
\begin{equation*}\label{definition of suitable permutation in the nonsplit case} 
(\HNslope{\genpolygon}{i})_{i \in S_\gensubvb} = (\HNslope{\HN(\gensubvb)}{i}) \quad \text{ and } \quad (\HNslope{\genpolygon}{i})_{i \in S_\genquotvb} = (\HNslope{\HN(\genquotvb)}{i}).
\end{equation*}
The $(S_{\genred{\gensubvb}}, S_{\genred{\genquotvb}})$-decomposition $\genred{\genextvb} \simeq \genred{\genextvb}_{\genred{\gensubvb}} \oplus \genred{\genextvb}_{\genred{\genquotvb}}$ and the commutative diagram \eqref{induction diagram of short exact sequences, necessity part} together yield
an $(S_\gensubvb, S_\genquotvb)$-decomposition $\genextvb \simeq \genextvb_\gensubvb \oplus \genextvb_\genquotvb$. Moreover, the injective map $\genextvb_\genquotvb \inj \genextvb \surj \genquotvb$ is an isomorphism at the generic point as we have $\rk(\genextvb_\genquotvb) = \rk(\genred{\genextvb}_{\genred{\genquotvb}}) = \rk(\genred{\genquotvb}) = \rk(\genquotvb)$. We then find from the short exact sequence \eqref{short exact sequence, necessity part integer slopes} that the map $\gensubvb \inj \genextvb \surj \genextvb/\genextvb_\genquotvb \cong \genextvb_\gensubvb$ is also an isomorphism at the generic point. Now Lemma \ref{necessary condition for HN polygon of subsheaves and quotients} yields inequalities
\begin{equation*}
\begin{aligned} 
\HNslope{\HN(\gensubvb)}{i} \leq \HNslope{\HN(\genextvb_\gensubvb)}{i} \quad& \text{ for } i = 1, \cdots, \rk(\gensubvb), \\ 
\HNslope{\HN(\genquotvb)}{j} \geq \HNslope{\HN(\genextvb_\genquotvb)}{j} \quad& \text{ for } j = 1, \cdots, \rk(\genquotvb).
\end{aligned}
\end{equation*}
We can rewrite these inequalities as
\begin{equation*}\label{slope dominance inequalities for induction, integer slopes}
\begin{aligned} 
\HNslope{\genpolygon}{i} \leq \HNslope{\HN(\genextvb)}{i}  \quad& \text{ for all } i \in S_\gensubvb, \\ 
\HNslope{\genpolygon}{j} \geq \HNslope{\HN(\genextvb)}{j}  \quad& \text{ for all } j \in S_\genquotvb.
\end{aligned}
\end{equation*}
In addition, since $\genred{\gensubvb}$ and $\genred{\genquotvb}$ are respectively subsheaves of $\gensubvb$ and $\genquotvb$, Lemma \ref{necessary condition for HN polygon of subsheaves and quotients} yields
\begin{equation*}
\begin{aligned} 
\HNslope{\HN(\genred{\gensubvb})}{i} \leq \HNslope{\HN(\gensubvb)}{i} \quad& \text{ for } i = 1, \cdots, \rk(\gensubvb) -1, \\ \HNslope{\HN(\genred{\genquotvb})}{j} \leq \HNslope{\HN(\genquotvb)}{j} \quad& \text{ for } j = 1, \cdots, \rk(\genquotvb).
\end{aligned}
\end{equation*}
As $\genred{\genpolygon}$ is $(S_{\genred{\gensubvb}}, S_{\genred{\genquotvb}})$-sorted, we find $\genpolygon \geq \HN(\genextvb)$. Hence we deduce by Lemma \ref{partition lemma for extension-permutations} that $\genpolygon$ is a $(\gensubvb, \genextvb, \genquotvb)$-permutation of $\HN(\gensubvb \oplus \genquotvb)$ with a $\genpolygon$-partition pair $(S_\gensubvb, S_\genquotvb)$ and an $(S_\gensubvb, S_\genquotvb)$-decomposition $\genextvb \simeq \genextvb_\gensubvb \oplus \genextvb_\genquotvb$, as illustrated by Figure \ref{suitable permutation in the nonsplit case}.
\end{proof}

\begin{figure}[H]
\begin{tikzpicture}[scale=0.5]	
		\coordinate (left) at (0, 0);
		\coordinate (p1) at (3, 3);
		\coordinate (p2) at (7, 4);
		\coordinate (p3) at (10, 3);
		\coordinate (right) at (11, 1);

		\coordinate (q1) at (1, 4);
		\coordinate (q2) at (3, 6);
		\coordinate (q3) at (5, 7);
		\coordinate (q4) at (7, 7);
		\coordinate (q5) at (10, 6);
		
		\coordinate (r1) at (1, 3);
		\coordinate (r2) at (3, 5);
		\coordinate (r3) at (5, 6);
		\coordinate (r4) at (8, 5);
				
		\draw[step=1cm,thick, color=mynicegreen] (left) -- (p1) --  (p2) -- (p3) -- (right);
		\draw[step=1cm,thick, color=red] (left) -- (q1) -- (q2);
		\draw[step=1cm,thick, color=violet] (r2) -- (r3) -- (r4) -- (p3);
		\draw[step=1cm,thick, color=orange] (left) -- (r1) -- (r2);
		\draw[step=1cm,thick, color=blue] (q2) -- (q3) -- (q4) -- (q5) -- (right);

		\draw [fill] (q1) circle [radius=0.05];		
		\draw [fill] (q2) circle [radius=0.05];		
		\draw [fill] (q3) circle [radius=0.05];		
		\draw [fill] (q4) circle [radius=0.05];		
		\draw [fill] (q5) circle [radius=0.05];		
		\draw [fill] (left) circle [radius=0.05];
		\draw [fill] (right) circle [radius=0.05];		
		
		\draw [fill] (p1) circle [radius=0.05];		
		\draw [fill] (p2) circle [radius=0.05];		
		\draw [fill] (p3) circle [radius=0.05];		
		
		\draw [fill] (r1) circle [radius=0.05];		
		\draw [fill] (r2) circle [radius=0.05];		
		\draw [fill] (r3) circle [radius=0.05];		
		\draw [fill] (r4) circle [radius=0.05];		
		
		\path (q5) ++(0.0, 0.5) node {\color{blue}$\gensubvb$};
		\path (p2) ++(0.0, -0.6) node {\color{mynicegreen}$\genextvb$};
		\path (q1) ++(-0.5, 0.2) node {\color{red}$\genquotvb$};
		\path (r1) ++(0.5, -0.2) node {\color{orange}$\genred{\genquotvb}$};
		\path (r4) ++(0.6, 0.2) node {\color{violet}$\genred{\gensubvb}$};
\end{tikzpicture}
\hspace{0.3cm}
\begin{tikzpicture}[scale=0.4]
        \pgfmathsetmacro{\textycoordinate}{5}
		\draw[->, line width=0.6pt] (0, \textycoordinate) -- (1.5,\textycoordinate);
		\draw (0,0) circle [radius=0.00];	
\end{tikzpicture}
\hspace{0.3cm}
\begin{tikzpicture}[scale=0.5]
		\coordinate (left) at (0, 0);
		\coordinate (p1) at (3, 3);
		\coordinate (p2) at (7, 4);
		\coordinate (p3) at (10, 3);
		\coordinate (right) at (11, 1);

		\coordinate (q1) at (1, 4);
		\coordinate (q2) at (3, 6);
		\coordinate (q3) at (5, 7);
		\coordinate (q4) at (7, 7);
		\coordinate (q5) at (10, 6);
		
		\coordinate (r1) at (1, 3);
		\coordinate (r2) at (3, 5);
		\coordinate (r3) at (5, 6);
		\coordinate (r4) at (8, 5);
		
		\coordinate (s1) at (q1);
		\coordinate (s2) at (3, 5);
		\coordinate (s3) at (q3);
		\coordinate (s4) at (q4);
		\coordinate (s5) at (q5);
		
		\coordinate (t1) at (r1);
		\coordinate (t2) at (3, 4);
		\coordinate (t3) at (r3);
		\coordinate (t4) at (r4);
				
		\draw[step=1cm,thick, color=mynicegreen] (left) -- (p1) --  (p2) -- (p3) -- (right);
		\draw[step=1cm,thick, color=red] (left) -- (s1);
		\draw[step=1cm,thick, color=orange] (left) -- (t1);
		\draw[step=1cm,thick, color=violet] (t1) -- (t2);
		\draw[step=1cm,thick, color=orange] (t2) -- (t3);
		\draw[step=1cm,thick, color=violet] (t3) -- (t4) -- (p3);
		
		\draw[step=1cm,thick, color=blue] (s1) -- (s2);
		\draw[step=1cm,thick, color=red] (s2) -- (s3);
		\draw[step=1cm,thick, color=blue] (s3) -- (s4) -- (s5) -- (right);
		
		\draw [fill] (left) circle [radius=0.05];
		\draw [fill] (right) circle [radius=0.05];		
		
		\draw [fill] (p1) circle [radius=0.05];		
		\draw [fill] (p2) circle [radius=0.05];	
		\draw [fill] (p3) circle [radius=0.05];	
		
		\draw [fill] (s1) circle [radius=0.05];		
		\draw [fill] (s2) circle [radius=0.05];		
		\draw [fill] (s3) circle [radius=0.05];		
		\draw [fill] (s4) circle [radius=0.05];	
		\draw [fill] (s5) circle [radius=0.05];	
		
		\draw [fill] (t1) circle [radius=0.05];		
		\draw [fill] (t2) circle [radius=0.05];		
		\draw [fill] (t3) circle [radius=0.05];		
		\draw [fill] (t4) circle [radius=0.05];			
		
\end{tikzpicture}
\caption{Construction of $\genpolygon$ in the case $\rk(\genred{\genquotvb}) = \rk(\genquotvb)$}\label{suitable permutation in the nonsplit case}
\end{figure}

\begin{theorem}\label{necessary conditions for HN polygon of extensions}
Let $\gensubvb$, $\genextvb$ and $\genquotvb$ be vector bundles on $\schff_{\genlocfield, \FFfield}$ with a short exact sequence
\begin{equation}\label{short exact sequence, necessity part} 
0 \longrightarrow \gensubvb \longrightarrow \genextvb \longrightarrow \genquotvb \longrightarrow 0.
\end{equation}
There exists 
a $(\gensubvb, \genextvb, \genquotvb)$-permutation of $\HN(\gensubvb \oplus \genquotvb)$. 
\end{theorem}

\begin{proof}
Take an integer $d \neq 0$ such that $d \lambda$ is an integer for each slope $\lambda$ of $\gensubvb$, $\genextvb$, and $\genquotvb$. Let $\genlocfield'$ be the unramified extension of $\genlocfield$ of degree $d$. Proposition \ref{unramified covers of FF curve} yields a projection map 
\[ \pi: \schff_{\genlocfield', \FFfield} \cong \schff_{\genlocfield, \FFfield} \times_{\Spec(\genlocfield)} \Spec(\genlocfield') \longrightarrow \schff_{\genlocfield, \FFfield}\]
such that every vector bundle $\genvb$ on $\schff_{\genlocfield, \FFfield}$ satisfies 
\[\HNslope{\HN(\pi^*\genvb)}{i} = d \cdot \HNslope{\HN(\genvb)}{i} \quad\quad \text{ for } i = 1, \cdots, \rk(\genvb).\]
In particular, all slopes of $\pi^*\gensubvb$, $\pi^*\genextvb$, and $\pi^*\genquotvb$ are integers. Moreover, since $\pi$ is evidently flat, the exact sequence \eqref{short exact sequence, necessity part} gives rise to a short exact sequence
\[ 0 \longrightarrow \pi^*\gensubvb \longrightarrow \pi^*\genextvb \longrightarrow \pi^*\genquotvb \longrightarrow 0.\]
By Proposition \ref{necessary conditions for HN polygon of extensions with integer slopes}, there exists 
a $(\pi^*\gensubvb, \pi^*\genextvb, \pi^*\genquotvb)$-permutation $\genpolygon'$ of $\HN(\pi^*\gensubvb \oplus \pi^*\genquotvb)$. Hence we get 
a $(\gensubvb, \genextvb, \genquotvb)$-permutation $\genpolygon$ of $\HN(\gensubvb \oplus \genquotvb)$ with
\[\HNslope{\genpolygon}{i} = \HNslope{\genpolygon'}{i}/d \quad\quad \text{ for } i = 1, \cdots, \rk(\genextvb),\]
thereby completing the proof. 
\end{proof}



\section{Classification theorems for extensions}

In this section, we establish two classification theorems for extensions of vector bundles on 
the Fargues-Fontaine curve. 

\begin{lemma}\label{minimum slopes of extension-permutations}
Let $\gensubvb$, $\genextvb$ and $\genquotvb$ be vector bundles on $\schff_{\genlocfield, \FFfield}$. Suppose that there exists a $(\gensubvb, \genextvb, \genquotvb)$-permutation $\genpolygon$ of $\HN(\gensubvb \oplus \genquotvb)$. 
\begin{enumerate}[label=(\arabic*)]
\item We have $\HN(\gensubvb \oplus \genquotvb) \geq \genpolygon \geq \HN(\genextvb)$
\smallskip

\item $\genquotvb^{< \mumin(\gensubvb)}$ is a direct summand of $\genextvb$. 
\end{enumerate}
\end{lemma}

\begin{proof}
The first statement is evident by the concavity of $\HN(\gensubvb \oplus \genquotvb)$. Let us now consider the second statement. Suppose for the sake of contradiction that $\genquotvb^{< \mumin(\gensubvb)}$ is not a direct summand of $\genextvb$. 
By the concavity of $\HN(\gensubvb \oplus \genquotvb)$, we find
\[ (\HNslope{\HN(\gensubvb \oplus \genquotvb)}{i})_{i > \rk(\genextvb) - \rk(\genquotvb^{< \mumin(\gensubvb)})} = (\HNslope{\HN(\genquotvb^{< \mumin(\gensubvb)})}{i}).\]
Take $\auxindex$ to be the largest integer with $\HNslope{\HN(\genextvb)}{\auxindex} \neq \HNslope{\HN(\gensubvb \oplus \genquotvb)}{\auxindex}$. Then we must have $\auxindex > \rk(\genextvb) - \rk(\genquotvb^{< \mumin(\gensubvb)})$ and $\HNslope{\HN(\genextvb)}{\auxindex} < \HNslope{\HN(\gensubvb \oplus \genquotvb)}{\auxindex}$. Let us choose a rational number $\mu$ with $\HNslope{\HN(\genextvb)}{\auxindex} < \mu <\HNslope{\HN(\gensubvb \oplus \genquotvb)}{\auxindex}$. By concavity of HN polygons we find
\[ \rk(\genextvb^{< \mu}) = \rk(\genextvb) - \auxindex \quad \text{ and } \quad \rk(\genquotvb^{<\mu}) >\rk(\genextvb) - \auxindex.\]
However, we must have $\rk(\genextvb^{< \mu}) \leq \rk(\genquotvb^{<\mu})$ as there is a $(\gensubvb, \genextvb, \genquotvb)$-permutation of $\HN(\gensubvb \oplus \genquotvb)$. We thus have a desired contradiction, thereby completing the proof. 
\end{proof}

\begin{prop}[{\cite[Theorem 1.1.2]{Arizona_extvb}}]\label{classification of extensions, semistable kernel and cokernel}
Let $\gensubvb$, $\genextvb$ and $\genquotvb$ be vector bundles on $\schff_{\genlocfield, \FFfield}$ such that $\gensubvb$ and $\genquotvb$ are semistable with $\mu(\gensubvb) \leq \mu(\genquotvb)$. There exists a short exact sequence 
\[ 0 \longrightarrow \gensubvb \longrightarrow \genextvb \longrightarrow \genquotvb \longrightarrow 0\]
if and only if we have $\HN(\gensubvb \oplus \genquotvb) \geq \HN(\genextvb)$. 
\end{prop}

\begin{remark}
While the cited result \cite[Theorem 1.1.2]{Arizona_extvb} does not explicitly consider the case where $\mu(\gensubvb)$ and $\mu(\genquotvb)$ are equal, this case follows immediately from Proposition \ref{basic properties of stable bundles}.
\end{remark}

\begin{theorem}\label{classification of extensions, semistable case}
Let $\gensubvb$, $\genextvb$ and $\genquotvb$ be vector bundles on $\schff_{\genlocfield, \FFfield}$ such that either $\gensubvb$ or $\genquotvb$ is semistable. There exists a short exact sequence 
\begin{equation}\label{short exact sequence, semistable case}
0 \longrightarrow \gensubvb \longrightarrow \genextvb \longrightarrow \genquotvb \longrightarrow 0
\end{equation}
if and only if there exists a $(\gensubvb, \genextvb, \genquotvb)$-permutation of $\HN(\gensubvb \oplus \genquotvb)$.
\end{theorem}

\begin{proof}
Let $\gensubvb^\vee$, $\genextvb^\vee$, and $\genquotvb^\vee$ respectively denote the duals of $\gensubvb$, $\genextvb$ and $\genquotvb$. 
By Proposition \ref{basic properties of stable bundles} and Theorem \ref{existence of HN decomp}, we observe that a vector bundle on $\schff_{\genlocfield, \FFfield}$ is semistable if and only if its dual is semistable, and also find that $\HN(\gensubvb \oplus \genquotvb)$ has a $(\gensubvb, \genextvb, \genquotvb)$-permutation if and only if $\HN(\gensubvb^\vee \oplus \genquotvb^\vee)$ has a $(\gensubvb^\vee, \genextvb^\vee, \genquotvb^\vee)$-permutation. 
Moreover, the existence of a short exact sequence \eqref{short exact sequence, semistable case} is equivalent to the existence of a short exact sequence
\[ 0 \longrightarrow \genquotvb^\vee \longrightarrow \genextvb^\vee \longrightarrow \gensubvb^\vee \longrightarrow 0. \]
 Hence we may assume without loss of generality that $\gensubvb$ is semistable. 

The necessity part is an immediate consequence of Theorem \ref{necessary conditions for HN polygon of extensions}. For the sufficiency part, we henceforth assume that there exists a $(\gensubvb, \genextvb, \genquotvb)$-permutation $\genpolygon$ of $\HN(\gensubvb \oplus \genquotvb)$. 
We also fix a $\genpolygon$-partition pair $(S_\gensubvb, S_\genquotvb)$ and assume in light of Lemma \ref{sorted suitable permutation} that $\genpolygon$ is $(S_\gensubvb, S_\genquotvb)$-sorted. 
Let us write $\numslope$ for the number of distinct slopes in $\HN(\genquotvb)$ and proceed by induction on $\numslope$. 

We first consider the base case where $\HN(\genquotvb)$ is a line segment. Note that $\genquotvb$ is semistable by Theorem \ref{existence of HN decomp}. If we have $\mu(\gensubvb) \leq \mu(\genquotvb)$, then Lemma \ref{minimum slopes of extension-permutations} and Proposition \ref{classification of extensions, semistable kernel and cokernel} together yield a desired short exact sequence \eqref{short exact sequence, semistable case}. If we have $\mu(\gensubvb) > \mu(\genquotvb)$, then Lemma \ref{minimum slopes of extension-permutations} implies that $\genextvb$ is isomorphic to $\gensubvb \oplus \genquotvb$ and thus gives rise to a desired (splitting) exact sequence \eqref{short exact sequence, semistable case}.

For the induction step, we assume from now on that $\HN(\genquotvb)$ has at least two distinct slopes. We have a decomposition
\[ \genquotvb \simeq \bigoplus_{i = 1}^\numslope \genquotvb_i\]
where $\genquotvb_1, \cdots, \genquotvb_\numslope$ are semistable with $\mu(\genquotvb_1) > \cdots > \mu(\genquotvb_\numslope)$, 
and 
thus obtain a decomposition 
\[\genquotvb \simeq \genred{\genquotvb} \oplus \genquotvb_\numslope\] 
where $\genred{\genquotvb}$ is a vector bundle on $\schff_{\genlocfield, \FFfield}$ with $r-1$ distinct slopes in $\HN(\genred{\genquotvb})$ and  $\mumin(\genred{\genquotvb}) > \mu(\genquotvb_r)$.

We consider the case where $\mu(\gensubvb)$ is greater than $\mu(\genquotvb_r)$. The readers may refer to Figure \ref{construction of aux bundles for the degenerate case} for a visual illustration of our argument. Lemma \ref{minimum slopes of extension-permutations} yields a decomposition $\genextvb \simeq \genred{\genextvb} \oplus \genquotvb_r$ for some vector bundle $\genred{\genextvb}$ over $\schff_{\genlocfield, \FFfield}$, and also implies that $\HN(\genextvb)$, $\genpolygon$, and $\HN(\gensubvb \oplus \genquotvb)$ must coincide on the interval $[\rk(\genred{\genextvb}), \rk(\genextvb)]$. Let us take $\genred{\genpolygon}$ to be the restriction of $\genpolygon$ on the interval $[0, \rk(\genred{\genextvb})]$. Then $\genred{\genpolygon}$ is a $(\gensubvb, \genred{\genextvb}, \genred{\genquotvb})$-permutation of $\HN(\gensubvb \oplus \genred{\genquotvb})$, and thus gives rise to a short exact sequence
\[ 0 \longrightarrow \gensubvb \longrightarrow \genred{\genextvb} \longrightarrow \genred{\genquotvb} \longrightarrow 0\]
by the induction hypothesis. Now we get a desired exact sequence \eqref{short exact sequence, semistable case} using the decompositions $\genextvb \simeq \genred{\genextvb} \oplus \genquotvb_r$ and $\genquotvb \simeq \genred{\genquotvb} \oplus \genquotvb_\numslope$.
\begin{figure}[H]
\begin{tikzpicture}[scale=0.5]	
		\pgfmathsetmacro{\rightx}{11}
		\pgfmathsetmacro{\righty}{1}
		
		\pgfmathsetmacro{\eonex}{3}
		\pgfmathsetmacro{\eoney}{3}
		\pgfmathsetmacro{\etwox}{4}
		\pgfmathsetmacro{\etwoy}{1}
		\pgfmathsetmacro{\ethreex}{3}
		\pgfmathsetmacro{\ethreey}{-1}
		
		\pgfmathsetmacro{\fonex}{2}		
		\pgfmathsetmacro{\ftwox}{3.5}
		\pgfmathsetmacro{\ftwoy}{1.5}
		\pgfmathsetmacro{\fthreex}{1}
		\pgfmathsetmacro{\fthreey}{-2}
		
		\pgfmathsetmacro{\pyoffset}{1}
		
		\pgfmathsetmacro{\dslope}{(\eoney+\pyoffset+\ftwoy+\fthreey-\righty)/(\eonex+\ftwox+\fthreex-\rightx)}
		
		\pgfmathsetmacro{\foney}{\eoney+\pyoffset-(\eonex-\fonex)*\dslope}		
				
		\coordinate (left) at (0, 0);
		\coordinate (p1) at (\eonex, \eoney);
		\coordinate (p2) at (\eonex+\etwox, \eoney+\etwoy);
		\coordinate (p3) at (\eonex+\etwox+\ethreex, \eoney+\etwoy+\ethreey);
		\coordinate (right) at (\rightx, \righty);

		\coordinate (q1) at (\fonex, \foney);
		\coordinate (q2) at (\fonex+\ftwox, \foney+\ftwoy);
		
		\draw[step=1cm,thick, color=mynicegreen] (left) -- (p1) --  (p2) -- (p3);
		\draw[step=1cm,thick, color=red] (left) -- (q1);
		\draw[step=1cm,thick, color=red] (q1) -- (q2);
		\draw[step=1cm,thick, color=blue] (q2) -- (p3);
		\draw[step=1cm,thick, color=orange] (p3) -- (right);

		\draw [fill] (left) circle [radius=0.05];
		\draw [fill] (right) circle [radius=0.05];		
		
		\draw [fill] (p1) circle [radius=0.05];		
		\draw [fill] (p2) circle [radius=0.05];	
		\draw [fill] (p3) circle [radius=0.05];	

		\draw [fill] (q1) circle [radius=0.05];		
		\draw [fill] (q2) circle [radius=0.05];	
		
		\pgfmathsetmacro{\vdx}{(\eonex+\etwox+\ethreex+\fonex+\ftwox)*0.5}	
		\pgfmathsetmacro{\vdy}{(\eoney+\etwoy+\ethreey+\foney+\ftwoy)*0.5}

		\path (p1) ++(0.4, -0.5) node {\color{mynicegreen}$\genred{\genextvb}$};
		\path (q1) ++(-0.5, -0.0) node {\color{red}$\genred{\genquotvb}$};
		\path (\vdx,\vdy) ++(0.0, 0.5) node {\color{blue}$\gensubvb$};
		\path (p3) ++(1.0, -0.8) node {\color{orange}$\genquotvb_r$};
\end{tikzpicture}
\hspace{0.3cm}
\begin{tikzpicture}[scale=0.4]
        \pgfmathsetmacro{\textycoordinate}{5}
		\draw[->, line width=0.6pt] (0, \textycoordinate) -- (1.5,\textycoordinate);
		\draw (0,0) circle [radius=0.00];	
\end{tikzpicture}
\hspace{0.3cm}
\begin{tikzpicture}[scale=0.5]
		\pgfmathsetmacro{\rightx}{11}
		\pgfmathsetmacro{\righty}{1}
		
		\pgfmathsetmacro{\eonex}{3}
		\pgfmathsetmacro{\eoney}{3}
		\pgfmathsetmacro{\etwox}{4}
		\pgfmathsetmacro{\etwoy}{1}
		\pgfmathsetmacro{\ethreex}{3}
		\pgfmathsetmacro{\ethreey}{-1}
		
		\pgfmathsetmacro{\fonex}{2}		
		\pgfmathsetmacro{\ftwox}{3.5}
		\pgfmathsetmacro{\ftwoy}{1.5}
		\pgfmathsetmacro{\fthreex}{1}
		\pgfmathsetmacro{\fthreey}{-2}
		
		\pgfmathsetmacro{\pyoffset}{1}
		
		\pgfmathsetmacro{\dslope}{(\eoney+\pyoffset+\ftwoy+\fthreey-\righty)/(\eonex+\ftwox+\fthreex-\rightx)}
		
		\pgfmathsetmacro{\foney}{\eoney+\pyoffset-(\eonex-\fonex)*\dslope}		
				
		\coordinate (left) at (0, 0);
		\coordinate (p1) at (\eonex, \eoney);
		\coordinate (p2) at (\eonex+\etwox, \eoney+\etwoy);
		\coordinate (p3) at (\eonex+\etwox+\ethreex, \eoney+\etwoy+\ethreey);
		\coordinate (right) at (\rightx, \righty);

		\coordinate (q1) at (\fonex, \foney);
		\coordinate (q2) at (\fonex+\ftwox, \foney+\ftwoy);
		
		\coordinate (s1) at (q1);
		\coordinate (s2) at (\eonex, \eoney+\pyoffset);
		\coordinate (s3) at (\eonex+\ftwox, \eoney+\ftwoy+\pyoffset);
		
		\draw[step=1cm,thick, color=mynicegreen] (left) -- (p1) --  (p2) -- (p3);
		\draw[step=1cm,thick, color=orange] (p3) --  (right);
		\draw[step=1cm,thick, color=blue] (s3) -- (p3);
		
		\draw[step=1cm,thick, color=red] (left) -- (q1);
		\draw[step=1cm,thick,dotted, color=red] (q1) -- (q2);
		\draw[step=1cm,thick,dotted, color=blue] (q2) -- (s3);
		
		\draw[step=1cm,thick, color=blue] (s1) -- (s2);
		\draw[step=1cm,thick, color=red] (s2) -- (s3);
		
		\draw [fill] (left) circle [radius=0.05];
		\draw [fill] (right) circle [radius=0.05];		
		
		\draw [fill] (p1) circle [radius=0.05];		
		\draw [fill] (p2) circle [radius=0.05];	
		\draw [fill] (p3) circle [radius=0.05];	

		\draw [fill] (q1) circle [radius=0.05];		
		\draw [fill] (q2) circle [radius=0.05];	
				
		\draw [fill] (s1) circle [radius=0.05];		
		\draw [fill] (s2) circle [radius=0.05];		
		\draw [fill] (s3) circle [radius=0.05];		
\end{tikzpicture}
\caption{Construction of $\genred{\genextvb}$ and $\genred{\genpolygon}$ in the case $\mu(\gensubvb) > \mu(\genquotvb_r)$}\label{construction of aux bundles for the degenerate case}
\end{figure}

It remains to consider the case where we have $\mu(\gensubvb) \leq \mu(\genquotvb_r)$. The readers may refer to Figure \ref{construction of aux bundles for the nondegenerate case} for a visual illustration of our argument. 
By Lemma \ref{minimum slopes of extension-permutations} we find $\mu(\gensubvb) = \mumin(\gensubvb \oplus \genquotvb) \leq \mumin(\genextvb)$. In addition, we have $\rk(\genextvb^{>\mu(\genquotvb_r)}) + \rk(\genquotvb_r) \leq \rk(\genextvb)$ as the surjective map $\genextvb \surj \genquotvb \surj \genquotvb_r$ factors through $\genextvb^{\leq \mu(\genquotvb_r)}$ by Proposition \ref{basic properties of stable bundles}.  
Take $\genred{\genextvb}$ to be the vector bundle on $\schff_{\genlocfield, \FFfield}$ such that $\HN(\genred{\genextvb} \oplus \genquotvb_r)$ is the upper convex hull of $\HN(\genextvb^{>\mu(\genquotvb_r)} \oplus \genquotvb_r)$ and $\HN(\genextvb)$. Let $\gensecred{\genextvb}$ denote the maximal common direct summand of $\genred{\genextvb}$ and $\genextvb$. 
Then we find
\begin{equation}\label{semistable classification E reduction decompositions} 
\genextvb \simeq \gensecred{\genextvb} \oplus \genthirdred{\genextvb} \quad \text{ and } \quad \genred{\genextvb} \simeq \gensecred{\genextvb} \oplus \auxsubvb
\end{equation}
where $\genthirdred{\genextvb}$ and $\auxsubvb$ are vector bundles on $\schff_{\genlocfield, \FFfield}$ with $\auxsubvb$ being semistable. By construction, we have $\HN(\genquotvb_r \oplus \auxsubvb) \geq \HN(\genthirdred{\genextvb})$ and $\mu(\genquotvb_r) \geq \mu(\auxsubvb)$. Hence Proposition \ref{classification of extensions, semistable kernel and cokernel} 
and the decompositions \eqref{semistable classification E reduction decompositions} together yield a short exact sequence
\begin{equation}\label{E-bar as a subbundle of E with minimal semistable quotient}
0 \longrightarrow \genred{\genextvb} \longrightarrow \genextvb \longrightarrow \genquotvb_r \longrightarrow 0.
\end{equation}
Let us now set
$S_{\genred{\genquotvb}}:= \{ i \in S_\genquotvb: \HNslope{\genpolygon}{i} \neq \mu(\genquotvb_r) \}$
and take $\genred{\genpolygon}$ to be the rationally tuplar polygon with
$(\HNslope{\genred{\genpolygon}}{i}) = (\HNslope{\genpolygon}{i})_{i \in S_{\genred{\genquotvb}} \,\cup \,S_\gensubvb}$; in other words, we obtain $\genred{\genpolygon}$ by removing all line segments of slope $\HNslope{\genpolygon}{i}$ from $\genpolygon$. Since $S_\genquotvb \backslash S_{\genred{\genquotvb}}$ does not contain any integer less than or equal to $\rk(\genextvb^{>\mu(\genquotvb_r)})$, the polygons $\genpolygon$ and $\genred{\genpolygon}$ coincide on the interval $[0, \rk(\genextvb^{>\mu(\genquotvb_r)})]$. 
In addition, as $\HN(\genred{\genextvb} \oplus \genquotvb_r)$ is the upper convex hull of $\HN(\genextvb^{>\mu(\genquotvb_r)} \oplus \genquotvb_r)$ and $\HN(\genextvb)$, we find
\[ \mu(\gensubvb) \leq \mumin(\genextvb) \leq \HNslope{\genred{\genextvb}}{i} \leq \mu(\genquotvb_r) = \mumin(\genquotvb) \quad\quad \text{ for } i = \rk(\genextvb^{>\mu(\genquotvb_r)})+1, \cdots, \rk(\genred{\genextvb}).\]
Hence we obtain a $(\gensubvb, \genred{\genextvb} \oplus \genquotvb_r, \genquotvb)$-permutation $\genpolygon'$ of $\HN(\gensubvb \oplus \genquotvb)$ by concatenating $\genred{\genpolygon_1}$, $\HN(\genquotvb_r)$, and $\genred{\genpolygon}_2$, where $\genred{\genpolygon}_1$ and $\genred{\genpolygon}_2$ are 
the rationally tuplar polygons with
\[ (\HNslope{\genred{\genpolygon}_1}{i}) = (\HNslope{\genred{\genpolygon}}{i})_{i \leq \rk(\genextvb^{>\mu(\genquotvb_r)})} \quad \text{ and } \quad  (\HNslope{\genred{\genpolygon}_2}{i}) = (\HNslope{\genred{\genpolygon}}{i})_{i > \rk(\genextvb^{>\mu(\genquotvb_r)})}.\]
It follows that $\genred{\genpolygon}$ is a $(\gensubvb, \genred{\genextvb}, \genred{\genquotvb})$-permutation of $\HN(\gensubvb \oplus \genred{\genquotvb})$; indeed, $\genred{\genpolygon}$ and $\genred{\genextvb}$ are respectively obtained from $\genextvb'$ and $\genred{\genextvb} \oplus \genquotvb_r$ by removing $\HN(\genquotvb_r)$ over the same interval. 
Hence we obtain a short exact sequence
\begin{equation}\label{E-bar as a semistable extension}
0 \longrightarrow \gensubvb \longrightarrow \genred{\genextvb} \longrightarrow \genred{\genquotvb} \longrightarrow 0
\end{equation}
by the induction hypothesis. Now the exact sequences \eqref{E-bar as a subbundle of E with minimal semistable quotient} and \eqref{E-bar as a semistable extension} together yield a commutative diagram of short exact sequences
\begin{equation*}
\begin{tikzcd}
0 \arrow[r]& \gensubvb \ar[equal]{r}\arrow[d, hookrightarrow]& \gensubvb \arrow[r]\arrow[d, "\alpha", hookrightarrow]& 0 \arrow[r]\arrow[d, hookrightarrow]& 0\\
0 \arrow[r]& \genred{\genextvb} \arrow[r]& \genextvb \arrow[r]& \genquotvb_r \arrow[r]& 0
\end{tikzcd}
\end{equation*}
which, by the snake lemma, induces a short exact sequence
\[0 \longrightarrow \genred{\genquotvb} \longrightarrow \coker(\alpha) \longrightarrow \genquotvb_r \longrightarrow 0.\]
Since this sequence is split by Proposition \ref{basic properties of stable bundles}, we obtain a desired exact sequence \eqref{short exact sequence, semistable case}.
\end{proof}

\begin{figure}[H]
\begin{tikzpicture}[scale=0.5]	
		\pgfmathsetmacro{\rightx}{10.5}
		\pgfmathsetmacro{\righty}{2.2}
		
		\pgfmathsetmacro{\eonex}{2}
		\pgfmathsetmacro{\eoney}{4}
		\pgfmathsetmacro{\etwox}{4}
		\pgfmathsetmacro{\etwoy}{1}
		\pgfmathsetmacro{\ethreex}{3}
		\pgfmathsetmacro{\ethreey}{-1}
		
		\pgfmathsetmacro{\fonex}{1.5}		
		\pgfmathsetmacro{\ftwox}{1.7}
		\pgfmathsetmacro{\ftwoy}{1.7}
		\pgfmathsetmacro{\fthreex}{3.5}
		\pgfmathsetmacro{\fthreey}{1.8}
		
		\pgfmathsetmacro{\pyoffset}{0.6}
		
		\pgfmathsetmacro{\dslope}{(\eoney+\pyoffset+\ftwoy+\fthreey-\righty)/(\eonex+\ftwox+\fthreex-\rightx)}
		
		\pgfmathsetmacro{\foney}{\eoney+\pyoffset-(\eonex-\fonex)*\dslope}		
				
		\coordinate (left) at (0, 0);
		\coordinate (p1) at (\eonex, \eoney);
		\coordinate (p2) at (\eonex+\etwox, \eoney+\etwoy);
		\coordinate (p3) at (\eonex+\etwox+\ethreex, \eoney+\etwoy+\ethreey);
		\coordinate (right) at (\rightx, \righty);

		\coordinate (q1) at (\fonex, \foney);
		\coordinate (q2) at (\fonex+\ftwox, \foney+\ftwoy);
		\coordinate (q3) at (\fonex+\ftwox+\fthreex, \foney+\ftwoy+\fthreey);
		
		\coordinate (r1) at (p1);
		\coordinate (r2) at (\eonex+\fthreex, \eoney+\fthreey);
		\coordinate (r3) at (p3);
		
		\draw[step=1cm,thick, color=mynicegreen] (left) -- (p1) --  (p2) -- (p3) -- (right);
		\draw[step=1cm,thick, color=red] (left) -- (q1);
		\draw[step=1cm,thick, color=red] (q1) -- (q2);
		\draw[step=1cm,thick, color=orange] (q2) -- (q3);
		\draw[step=1cm,thick, color=orange] (r1) -- (r2);
		\draw[step=1cm,thick, color=violet] (r2) -- (r3);
		
		\draw[step=1cm,thick, color=blue] (q3) -- (right);
		
		\draw [fill] (left) circle [radius=0.05];
		\draw [fill] (right) circle [radius=0.05];		
		
		\draw [fill] (p1) circle [radius=0.05];		
		\draw [fill] (p2) circle [radius=0.05];	
		\draw [fill] (p3) circle [radius=0.05];	

		\draw [fill] (q1) circle [radius=0.05];		
		\draw [fill] (q2) circle [radius=0.05];	
		\draw [fill] (q3) circle [radius=0.05];	
		
		\draw [fill] (r1) circle [radius=0.05];		
		\draw [fill] (r2) circle [radius=0.05];	
		\draw [fill] (r3) circle [radius=0.05];	
		
		\pgfmathsetmacro{\vdx}{(\fonex+\ftwox+\fthreex+\rightx)*0.5}	
		\pgfmathsetmacro{\vdy}{(\foney+\ftwoy+\fthreey+\righty)*0.5}
		\pgfmathsetmacro{\vfthreex}{\fonex+\ftwox+\fthreex*0.5}
		\pgfmathsetmacro{\vfthreey}{\foney+\ftwoy+\fthreey*0.5}
		\pgfmathsetmacro{\vkx}{\eonex+(\etwox+\ethreex+\fthreex)*0.5}	
		\pgfmathsetmacro{\vky}{\eoney+(\etwoy+\ethreey+\fthreey)*0.5}	

		\path (\vdx, \vdy) ++(0.5, 0.2) node {\color{blue}$\gensubvb$};
		\path (p2) ++(0.0, -0.6) node {\color{mynicegreen}$\genextvb$};
		\path (q1) ++(-0.5, 0.2) node {\color{red}$\genred{\genquotvb}$};
		\path (\vfthreex, \vfthreey) ++(-0.2, 0.5) node {\color{orange}$\genquotvb_r$};
		\path (\vkx, \vky) ++(0.0, 0.5) node {\color{violet}$\auxsubvb$};
\end{tikzpicture}
\hspace{0.3cm}
\begin{tikzpicture}[scale=0.4]
        \pgfmathsetmacro{\textycoordinate}{5}
		\draw[->, line width=0.6pt] (0, \textycoordinate) -- (1.5,\textycoordinate);
		\draw (0,0) circle [radius=0.00];	
\end{tikzpicture}
\hspace{0.3cm}
\begin{tikzpicture}[scale=0.5]
		\pgfmathsetmacro{\rightx}{10.5}
		\pgfmathsetmacro{\righty}{2.2}
		
		\pgfmathsetmacro{\eonex}{2}
		\pgfmathsetmacro{\eoney}{4}
		\pgfmathsetmacro{\etwox}{4}
		\pgfmathsetmacro{\etwoy}{1}
		\pgfmathsetmacro{\ethreex}{3}
		\pgfmathsetmacro{\ethreey}{-1}
		
		\pgfmathsetmacro{\fonex}{1.5}		
		\pgfmathsetmacro{\ftwox}{1.7}
		\pgfmathsetmacro{\ftwoy}{1.7}
		\pgfmathsetmacro{\fthreex}{3.5}
		\pgfmathsetmacro{\fthreey}{1.8}
		
		\pgfmathsetmacro{\pyoffset}{0.6}
		
		\pgfmathsetmacro{\dslope}{(\eoney+\pyoffset+\ftwoy+\fthreey-\righty)/(\eonex+\ftwox+\fthreex-\rightx)}
		
		\pgfmathsetmacro{\foney}{\eoney+\pyoffset-(\eonex-\fonex)*\dslope}		
				
		\coordinate (left) at (0, 0);
		\coordinate (p1) at (\eonex, \eoney);
		\coordinate (p2) at (\eonex+\etwox, \eoney+\etwoy);
		\coordinate (p3) at (\eonex+\etwox+\ethreex, \eoney+\etwoy+\ethreey);
		\coordinate (right) at (\rightx, \righty);

		\coordinate (q1) at (\fonex, \foney);
		\coordinate (q2) at (\fonex+\ftwox, \foney+\ftwoy);
		\coordinate (q3) at (\fonex+\ftwox+\fthreex, \foney+\ftwoy+\fthreey);
		
		\coordinate (r1) at (p1);
		\coordinate (r2) at (\eonex+\fthreex, \eoney+\fthreey);
		\coordinate (r3) at (p3);
		
		\coordinate (s1) at (q1);
		\coordinate (s2) at (\eonex, \eoney+\pyoffset);
		\coordinate (s3) at (\eonex+\fthreex, \eoney+\fthreey+\pyoffset);
		\coordinate (s4) at (\eonex+\fthreex+\ftwox, \eoney+\fthreey+\ftwoy+\pyoffset);
		
		\draw[step=1cm,thick,dotted, color=mynicegreen] (p1) --  (p2) -- (p3);
		
		\draw[step=1cm,thick, color=red] (left) -- (q1);
		\draw[step=1cm,thick,dotted, color=red] (q1) -- (q2);
		\draw[step=1cm,thick,dotted, color=orange] (q2) -- (q3);
		\draw[step=1cm,thick, color=orange] (r1) -- (r2);
		
		\draw[step=1cm,thick, color=mydarkteal] (left) -- (p1);
		\draw[step=1cm,thick, color=mydarkteal] (r2) -- (r3) -- (right);
		\draw[step=1cm,thick, color=orange] (s2) -- (s3);
		\draw[step=1cm,thick, color=red] (s3) -- (s4);
		\draw[step=1cm,thick, color=blue] (s4) -- (right);
		
		\draw[step=1cm,thick,dotted, color=blue] (q3) -- (s4);
		\draw[step=1cm,thick, color=blue] (s1) -- (s2);
		
		\draw [fill] (left) circle [radius=0.05];
		\draw [fill] (right) circle [radius=0.05];		
		
		\draw [fill] (p1) circle [radius=0.05];		
		\draw [fill] (p2) circle [radius=0.05];	
		\draw [fill] (p3) circle [radius=0.05];	

		\draw [fill] (q1) circle [radius=0.05];		
		\draw [fill] (q2) circle [radius=0.05];	
		\draw [fill] (q3) circle [radius=0.05];	
		
		\draw [fill] (r1) circle [radius=0.05];		
		\draw [fill] (r2) circle [radius=0.05];	
		\draw [fill] (r3) circle [radius=0.05];	
				
		\draw [fill] (s1) circle [radius=0.05];		
		\draw [fill] (s2) circle [radius=0.05];		
		\draw [fill] (s3) circle [radius=0.05];		
		\draw [fill] (s4) circle [radius=0.05];			
		
		\path (r3) ++(-0.4, -0.4) node {\color{mydarkteal}$\genred{\genextvb}$};
\end{tikzpicture}
\caption{Construction of $\genred{\genextvb}$ and $\genred{\genpolygon}$ in the case $\mu(\gensubvb) \leq \mu(\genquotvb_r)$}\label{construction of aux bundles for the nondegenerate case}
\end{figure}

\begin{theorem}\label{classification of extensions, general case}
Let $\gensubvb$, $\genextvb$ and $\genquotvb$ be vector bundles on $\schff_{\genlocfield, \FFfield}$. Let $\genquotvb_1, \cdots, \genquotvb_\numslope$ denote the semistable vector bundles on $\schff_{\genlocfield, \FFfield}$ such that $\HN(\genquotvb_i)$ represents the $i$-th line segment in $\HN(\genquotvb)$. 
There exists a short exact sequence 
\begin{equation}\label{short exact sequence, general case}
0 \longrightarrow \gensubvb \longrightarrow \genextvb \longrightarrow \genquotvb \longrightarrow 0
\end{equation}
if and only if the following equivalent conditions are satisfied:
\begin{enumerate}[label=(\roman*)]
\item\label{extension and lift of HN filtration} There exists a filtration
\[ \gensubvb = \genextvb_0 \subset \genextvb_1 \subset \cdots \subset \genextvb_r = \genextvb\]
with $\genextvb_i/\genextvb_{i-1} \simeq \genquotvb_i$ for each $i = 1, \cdots, r$. 
\smallskip

\item\label{extension and sequence of E-permutations} There exists a sequence of vector bundles $\gensubvb = \genextvb_0, \genextvb_1, \cdots, \genextvb_r = \genextvb$ such that the polygon $\HN(\genextvb_{i-1} \oplus \genquotvb_i)$ has an $(\genextvb_{i-1}, \genextvb_i, \genquotvb_i)$-permutation for each $i = 1, \cdots, r$. 
\end{enumerate}
\end{theorem}

\begin{proof}
The equivalence of the conditions \ref{extension and lift of HN filtration} and \ref{extension and sequence of E-permutations} is evident by Theorem \ref{classification of extensions, semistable case}. In addition, the sufficiency part is an immediate consequence of Proposition \ref{basic properties of stable bundles}. For the necessity part, we henceforth assume that there exists an exact sequence \eqref{short exact sequence, general case}. 

Let us proceed by induction on $r$. The assertion is trivial for $r = 0$. For the induction step, we now assume that $r$ is not zero. Let us set
\[ \genred{\genquotvb} :=  \bigoplus_{i = 1}^{\numslope-1} \genquotvb_i.\]
Then we have $\genquotvb \simeq \genred{\genquotvb} \oplus \genquotvb_\numslope$. Take $\genextvb_{r-1}$ to be the kernel of the map $\genextvb \surj \genquotvb \surj \genquotvb_r$. By construction, $\genextvb_{r-1}$ contains $\gensubvb$ as a subsheaf. Hence we get a commutative diagram of short exact sequences
\begin{equation*}
\begin{tikzcd}
0 \arrow[r]& \gensubvb \arrow[r]\arrow[d, hookrightarrow]& \genextvb \arrow[r]\ar[d, equal]& \genquotvb \arrow[r]\arrow[d, twoheadrightarrow]& 0\\
0 \arrow[r]& \genextvb_{r-1} \arrow[r]& \genextvb \arrow[r]& \genquotvb_r \arrow[r]& 0
\end{tikzcd}
\end{equation*}
and consequently find by the snake lemma that the cokernel of the left vertical map is isomorphic to $\genred{\genquotvb}$. The desired assertion now follows by the induction hypothesis. 
\end{proof}

\begin{remark}
We may check the condition \ref{extension and sequence of E-permutations} in a finite amount of time. In fact, if we start with $i = r$ and inductively proceed with descending indices, we get finitely many candidates for each $\genextvb_{i-1}$ just by the following simple observations:
\begin{enumerate}[label=(\alph*)]
\item $\HN(\genextvb_{i-1})$ and $\HN(\genextvb_i)$ coincide on the interval $[0, \rk(\genextvb_i^{> \mu(\genquotvb_i)})]$. 
\smallskip

\item We have $\HN(\genquotvb_i \oplus \genextvb_{i-1}^{\leq \mu(\genquotvb_i)}) \geq \HN(\genextvb_i^{\leq \mu(\genquotvb_i)})$. 
\smallskip

\item All breakpoints in $\HN(\genextvb_{i-1})$ are integer points. 
\end{enumerate}
It is also worthwhile to note that the filtration in the condition \ref{extension and lift of HN filtration} lifts the Harder-Narasimhan filtration of $\genquotvb$. Our proof indeed shows that, for any Harder-Narasimhan category with splitting Harder-Narasimhan filtrations, every extension between two arbitrary objects $\gensubvb$ and $\genquotvb$ gives rise to a filtration that lifts the Harder-Narasimhan filtration of $\genquotvb$. 
\end{remark}

\begin{example}\label{counter example of classification theorem without additional assumption}
We present an example showing that Theorem \ref{classification of extensions, semistable case} does not hold without the semistability assumption on either $\gensubvb$ or $\genquotvb$. 
For ease of notation, we will write $\trivbundle(\lambda) = \trivbundle[\genlocfield, \FFfield](\lambda)$ for each $\lambda \in \Q$. 
Let us take
\[ \gensubvb:= \trivbundle(1/3) \oplus \trivbundle(-1), \quad \genextvb:= \trivbundle(3/2) \oplus \trivbundle(3/7), \quad \genquotvb:= \trivbundle(3) \oplus \trivbundle(3/4).\]
Then it is not hard to verify the following facts:
\begin{enumerate}[label=(\alph*)]
\item $\genextvb_1= \trivbundle(3/2) \oplus \trivbundle^{\oplus 3}$ is the only vector bundle on $\schff_{\genlocfield, \FFfield}$ such that $\HN(\genextvb_1 \oplus \trivbundle(3/4))$ has an $(\genextvb_1, \genextvb, \trivbundle(3/4))$-permutation.
\smallskip

\item $\HN(\gensubvb \oplus \trivbundle(3))$ does not have a $(\gensubvb, \genextvb_1, \trivbundle(3))$-permutation.
\end{enumerate}
Hence Theorem \ref{classification of extensions, general case} implies that there does not exist a short exact sequence
\[0 \longrightarrow \gensubvb \longrightarrow \genextvb \longrightarrow \genquotvb \longrightarrow 0,\]
while $\HN(\gensubvb \oplus \genquotvb)$ has a $(\gensubvb, \genextvb, \genquotvb)$-permutation $\genpolygon$ with
\[ (\HNslope{\genpolygon}{i}) = (3, 1/3, 1/3, 1/3, 3/4, 3/4, 3/4, 3/4, -1)\]
as illustrated in Figure \ref{example with E-permutation and no extensions}. 
\begin{figure}[H]
\begin{tikzpicture}[scale=0.5]	
		\coordinate (left) at (0, 0);
		\coordinate (p1) at (2, 3);
		\coordinate (right) at (9, 6);

		\coordinate (q1) at (1, 3);
		\coordinate (q2) at (5, 6);
		\coordinate (q3) at (8, 7);
				
		\draw[step=1cm,thick, color=mynicegreen] (left) -- (p1) -- (right);
		\draw[step=1cm,thick, color=red] (left) -- (q1) -- (q2);
		\draw[step=1cm,thick, color=blue] (q2) -- (q3) -- (right);

		\draw [fill] (q1) circle [radius=0.05];		
		\draw [fill] (q2) circle [radius=0.05];		
		\draw [fill] (q3) circle [radius=0.05];			
		\draw [fill] (left) circle [radius=0.05];
		\draw [fill] (right) circle [radius=0.05];		
		
		\draw [fill] (p1) circle [radius=0.05];		
		
		\path (q3) ++(0.2, 0.6) node {\color{blue}$\gensubvb$};
		\path (p1) ++(2.2, 0.4) node {\color{mynicegreen}$\genextvb$};
		\path (q1) ++(-0.5, 0.2) node {\color{red}$\genquotvb$};
\end{tikzpicture}
\hspace{0.3cm}
\begin{tikzpicture}[scale=0.4]
        \pgfmathsetmacro{\textycoordinate}{5}
		\draw[->, line width=0.6pt] (0, \textycoordinate) -- (1.5,\textycoordinate);
		\draw (0,0) circle [radius=0.00];	
\end{tikzpicture}
\hspace{0.3cm}
\begin{tikzpicture}[scale=0.5]
		\coordinate (left) at (0, 0);
		\coordinate (p1) at (2, 3);
		\coordinate (right) at (9, 6);

		\coordinate (q1) at (1, 3);
		\coordinate (q2) at (5, 6);
		\coordinate (q3) at (8, 7);
		
		\coordinate (r1) at (q1);
		\coordinate (r2) at (4, 4);
		\coordinate (r3) at (8, 7);
		\coordinate (r4) at (right);
				
		\draw[step=1cm,thick, color=mynicegreen] (left) -- (p1) -- (right);
		\draw[step=1cm,thick, color=red] (left) -- (q1);
		\draw[step=1cm,thick,dotted, color=red] (q1) -- (q2);
		\draw[step=1cm,thick,dotted, color=blue] (q2) -- (q3);
		\draw[step=1cm,thick, color=blue] (q3) -- (right);
		
		\draw[step=1cm,thick, color=blue] (r1) -- (r2);
		\draw[step=1cm,thick, color=red] (r2) -- (r3);
		\draw[step=1cm,thick, color=blue] (r3) -- (r4);

		\draw [fill] (q1) circle [radius=0.05];		
		\draw [fill] (q2) circle [radius=0.05];		
		\draw [fill] (q3) circle [radius=0.05];			
		\draw [fill] (left) circle [radius=0.05];
		\draw [fill] (right) circle [radius=0.05];		
		
		\draw [fill] (p1) circle [radius=0.05];		

		\draw [fill] (r1) circle [radius=0.05];		
		\draw [fill] (r2) circle [radius=0.05];		
		\draw [fill] (r3) circle [radius=0.05];		
		\draw [fill] (r4) circle [radius=0.05];			
		
\end{tikzpicture}
\caption{A $(\gensubvb, \genextvb, \genquotvb)$-permutation of $\HN(\gensubvb \oplus \genquotvb)$ in Example \ref{counter example of classification theorem without additional assumption}}\label{example with E-permutation and no extensions}
\end{figure}
\end{example}

\begin{remark}
Example \ref{counter example of classification theorem without additional assumption} suggests that the condition \ref{extension and sequence of E-permutations} in Theorem \ref{classification of extensions, general case} is unlikely to have an equivalent statement which is easy to check in the general case. On the other hand, such a statement exists under some additional assumptions. 
\begin{enumerate}[label=(\arabic*)]
\item If all slopes in $\HN(\gensubvb)$, $\HN(\genextvb)$, and $\HN(\genquotvb)$ are integers, then the condition \ref{extension and sequence of E-permutations} in Theorem \ref{classification of extensions, general case} is satisfied if and only if $\HN(\gensubvb \oplus \genquotvb)$ has a $(\gensubvb, \genextvb, \genquotvb)$-permutation $\genpolygon$. 
\smallskip

\item If $\genextvb$ is semistable, then the condition \ref{extension and sequence of E-permutations} in Theorem \ref{classification of extensions, general case} is satisfied if and only if we have $\mumax(\gensubvb) \leq \mu(\genextvb) \leq \mumin(\genquotvb)$.
\end{enumerate}
The first statement can be proved by an induction argument similar to the proof of Theorem \ref{classification of extensions, semistable case}. The second statement immediately follows from the previous result of the author \cite[Theorem 1.1.1]{Hong_extvbss}. We also note that the first statement is comparable to the main result of Schlesinger \cite{Schlesinger_extvbonP1} which classifies all extension of two given vector bundles on $\PP^1$. 
\end{remark}

\bibliographystyle{amsalpha}

\bibliography{Bibliography}
	
\end{document}